\documentclass[journal,twoside]{IEEEtran}

\usepackage{cite}
\usepackage{textcomp}
\usepackage{graphicx}
\usepackage{color}

\usepackage[utf8]{inputenc}
\usepackage{amsopn,amsfonts,amssymb,amsmath}

\usepackage{amsthm}
\usepackage{xspace}
\usepackage{bm,fixmath}
\usepackage{hyperref}

\usepackage{accents}

\makeatletter
\let\cl@chapter\undefined
\makeatletter
\usepackage[noabbrev]{cleveref}

\usepackage{algorithm}
\usepackage{algpseudocode}
\algrenewcommand\algorithmicrequire{\textbf{Input:}}

\newtheorem{definition}{Definition}
\newtheorem{remark}{Remark}

\newtheorem{proposition}{Proposition}

\newtheorem{assumption}{Assumption}
\newtheorem{lemma}{Lemma}
\newtheorem{corollary}{Corollary}

\crefname{appsec}{Appendix}{Appendices}
\crefname{assumption}{Assumption}{Assumptions}
\crefname{corollary}{Corollary}{Corollaries}
\crefname{algorithm}{Algorithm}{Algorithms}
\crefname{proposition}{Proposition}{Propositions}
\crefname{lemma}{Lemma}{Lemmas}
\crefname{definition}{Definition}{Definitions}
\crefname{remark}{Remark}{Remarks}
\crefname{figure}{Figure}{Figures}
\crefname{table}{Table}{Tables}
\crefformat{equation}{(#2#1#3)}

\DeclareMathOperator{\fix}{fix}
\DeclareMathOperator{\diam}{diam}

\newcommand{\norm}[1]{\left\lVert#1\right\rVert}
\newcommand{\nnorm}[1]{{\left\vert\left\vert\left\vert#1\right\vert\right\vert\right\vert}}
\newcommand{\llangle}{\langle\mkern-2.5mu\langle}
\newcommand{\rrangle}{\rangle\mkern-2.5mu\rangle}

\newcommand{\ev}[2][]{\mathbb{E}_{#1}\left[#2\right]}
\newcommand{\pr}[1]{\mathbb{P}\left[#1\right]}

\newcommand{\e}{\mathbold{e}}
\newcommand{\uv}{\mathbold{u}}
\newcommand{\x}{\mathbold{x}}
\newcommand{\y}{\mathbold{y}}

\newcommand{\0}{\boldsymbol{0}}

\newcommand{\I}{\mathcal{I}}
\newcommand{\T}{\mathcal{T}}

\newcommand{\D}{\mathbb{D}}
\newcommand{\Is}{\mathbb{I}}
\newcommand{\Hs}{\mathbb{H}}
\newcommand{\N}{\mathbb{N}}
\newcommand{\R}{\mathbb{R}}

\newcommand{\km}{Krasnosel'ski\u{\i}-Mann\xspace}
\newcommand{\sw}{\mathrm{subW}}

\newcommand{\ubar}[1]{\underaccent{\bar}{#1}}
\newcommand{\pmin}{\ubar{p}}
\newcommand{\pmax}{\bar{p}}
\newcommand{\lmin}{\ubar{\lambda}}
\newcommand{\lmax}{\bar{\lambda}}

\begin{document}

\title{A Stochastic Operator Framework for Optimization and Learning\\with Sub-Weibull Errors}

\author{Nicola Bastianello, Liam Madden, Ruggero Carli, Emiliano Dall'Anese
\thanks{N. Bastianello is with the School of Electrical Engineering and Computer Science, and Digital Futures, KTH Royal Institute of Technology, Sweden. {\tt\small nicolba@kth.se}.}
\thanks{L. Madden is with the Department of Electrical and Computer Engineering, University of British Columbia, Vancouver, BC, CA. {\tt\small liam@ece.ubc.ca}}
\thanks{R. Carli is with the Department of Information Engineering (DEI), University of Padova, Italy. {\tt\small carlirug@dei.unipd.it}.}
\thanks{E. Dall'Anese is with the Department of Electrical, Computer and Energy Engineering, and affiliate faculty of the Department of Applied Mathematics, University of Colorado Boulder, Boulder, CO, US. {\tt\small emiliano.dallanese@colorado.edu}}
}

\maketitle

\begin{abstract}
This paper proposes a framework to study the convergence of stochastic optimization and learning algorithms. The framework is modeled over the different challenges that these algorithms pose, such as (i) the presence of random additive errors (\textit{e.g.} due to stochastic gradients), and (ii) random coordinate updates (\textit{e.g.} due to asynchrony in distributed set-ups). The paper covers  both convex and strongly convex problems, and it also analyzes online scenarios, involving changes in the data and costs.
The paper relies on interpreting stochastic algorithms as the iterated application of stochastic operators, thus allowing us to use the powerful tools of operator theory. In particular, we consider operators characterized by additive errors with sub-Weibull distribution (which parameterize a broad class of errors by their tail probability), and random updates.
In this framework we derive convergence results in mean and in high probability, by providing bounds to the distance of the current iteration from a solution of the optimization or learning problem.
The contributions are discussed in light of federated learning applications.
\end{abstract}


%---------------------------------------------------------------------------------------------
\section{Introduction}\label{sec:introduction}
Recent technological advances in a range of disciplines -- from machine learning, to data-driven optimization and control,
with their applications to smart power grids, traffic networks,  healthcare, etc. -- introduced a wide set of challenges to the implementation and analysis of optimization and learning algorithms. As a motivating example, we elaborate on these challenges in the context of  federated learning~\cite{li_federated_2020,gafni_federated_2022}, which was designed to allow a set of agents to cooperatively train a model without the need to directly share data.
Due to the distributed set-up, algorithms designed in this framework need to deal with \textit{asynchrony, and limited or unreliable communications}. This is similarly a challenge in any application where multi-agent systems are deployed, such as distributed optimization and parallel computing \cite{peng_coordinate_2016,bianchi_coordinate_2016,bastianello_asynchronous_2020,salzo_parallel_2021}.
Additionally, due to the size of available data-sets, the agents may need to resort to the use of \textit{stochastic gradients}, which are computed on a sub-set of the available data \cite{dixit_online_2019,liu_primer_2020,bottou2018optimization}. Using stochastic gradients reduces the computational cost but may lead to less accuracy.
Alternatively, when gradients are not directly accessible, for example due to sparse users' feedback \cite{simonetto_personalized_2021}, agents may need to approximate them from functional evaluations ($0$-th order gradients) \cite{duchi_optimal_2015,nesterov_random_2017,berahas_theoretical_2022}.
Finally, the data on which the agents train their model may change over time, due to changes in the phenomenon being observed or new data arriving in real time \cite{shalev-shwartz_online_2011,dallanese_convergence_2019,simonetto_time-varying_2020,hauswirth_timescale_2021}. This turns the problem into an \textit{online learning} problem, with the agents now tracking the optimal model as it changes in response to changes in the data.

Abstracting away from this example, in many applications of optimization and learning we need to design and analyze algorithms that are \textit{provably robust to different sources of stochasticity and changes in the data}. In particular, in this paper we focus on the analysis side of this challenge by leveraging the formalism of \textit{operator theory}.
Operator theory has been shown to offer a valuable tool for the analysis of algorithms, both in traditional static settings \cite{ryu_primer_2016,combettes_monotone_2018,bauschke_convex_2017,davis_convergence_2016} and in online settings \cite{shalev-shwartz_online_2011,simonetto_time-varying_2017,simonetto_time-varying_2020,Jadbabaie2015}. The key insight is that algorithmic steps can be interpreted as operators (or maps when working in an Euclidean space), with fixed points of the operators coinciding with  optimal solutions of the optimization or learning problem \cite{bauschke_convex_2017,combettes_monotone_2018,davis_convergence_2016}. This link gives access to powerful tools to characterize the convergence of existing algorithms \cite{bauschke_convex_2017,ryu_primer_2016,davis_convergence_2016}, and it further inspires the design of new ones, \emph{e.g.} \cite{beck_fast_2009,davis_three-operator_2017,themelis_supermann_2019}.

However, as discussed above, in many applications of interest the algorithms we apply are stochastic. To analyze their convergence, we will therefore make use of the formalism of \textit{stochastic operators}, which we review in the following.
One source of stochasticity is that of random coordinate updates, in which only some part of the operator is applied at any given iteration \cite{combettes_stochastic_2015,combettes_stochastic_2019,berinde_iterative_2007,bianchi_coordinate_2016,peng_arock_2016,salzo_parallel_2021}. This class of operators model distributed algorithms in which only some agents, \textit{e.g.} due to asynchrony, perform an update and hence apply the corresponding part of the algorithm/operator.
Operators subject to additive errors, \textit{e.g.} due to the use of stochastic gradients, have also been studied, see \cite{berinde_iterative_2007,borkar_concentration_2021,combettes_stochastic_2015,combettes_stochastic_2019,dixit_online_2019}. It is typical to handle stochastic errors in the algorithmic map or operator by either assuming that their norm vanishes asymptotically, or by multiplying them by a vanishing parameter. While this choice allows one to prove almost sure convergence, in many practical applications the additive error is not (or cannot be made) vanishing, especially in an online scenario.

To delineate our contribution in the context of the discussion above,  we analyze algorithms that can be characterized by the stochastic iteration (formalized in \cref{sec:framework}):
\begin{equation}\label{eq:bp-prototype}
    x_i^{\ell+1} = \begin{cases}
        \T_i \x^\ell + e_i^\ell & \text{w.p.} \ p_i \\
        x_i^\ell & \text{w.p.} \ 1 - p_i
    \end{cases}, \quad \ell \in \N
\end{equation}
where $\T \x = (\T_1 \x, \ldots, \T_n \x)$ is an operator, $\x = (x_1, \ldots, x_n)$, and $e_i^\ell$ is an additive error.
The stochasticity in~\cref{eq:bp-prototype} has two sources: \emph{(i)} each ``coordinate'' of the operator is updated with probability $p_i$ (\emph{e.g.} due to asynchrony), and \emph{(ii)} the update is performed using an inexact version of the operator (\textit{e.g.} stochastic gradients). Additionally, we will be interested in the \emph{(iii)} \emph{online} scenario in which the operator changes over time to model applications in online optimization and learning; that is, at time $\ell+1$ we apply $\T_i^{\ell+1}$. The iteration~\cref{eq:bp-prototype} will be studied under the assumption that the operator is either contractive or averaged (in which case~\cref{eq:bp-prototype} can be seen as a \emph{stochastic \km}).

As in \emph{e.g.}, \cite{combettes_stochastic_2015,combettes_stochastic_2019}, we model inexact operators as operators subject to the additive errors $e_i$. However, differently from previous results, in this paper we perform our analysis for errors with \textit{sub-Weibull distributions}. That is, the norm of $e_i$ satisfies
\begin{equation}\label{eq:tail-bound-prototype}
    \pr{\norm{e_i} \geq \epsilon} \leq 2 \exp\left( - \left( \frac{\epsilon}{\nu} \right)^{1 / \theta} \right), \quad \forall \,\, \epsilon > 0
\end{equation}
for some parameters $\theta \geq 0$, $\nu > 0$. The class of sub-Weibull r.v.s allows  us to consider distributions for the additive errors that may have heavy tails \cite{vladimirova_subweibull_2020,zhang_concentration_2021,kuchibhotla_moving_2022,wong_lasso_2020}; this class is also general and it includes the  sub-Gaussian and sub-exponential classes as sub-cases, as well as random variable whose distribution has a finite support~\cite{boucheron_concentration_2013,vershynin_high-dimensional_2018}.  As explained shortly, the sub-Weibull model will allow us to provide high probability bounds for the convergence of~\cref{eq:bp-prototype} by deriving pertinent concentration results. As defined in~\cref{eq:tail-bound-prototype}, sub-Weibull random variables are characterized by a tail parameter $\theta$, that defines the decay rate of the tails. Besides the convenient properties of the sub-Weibull class (\emph{e.g.} closure under scaling, sum, and product) that allow for a unified theoretical analysis, there is a growing body of work showing that heavy-tailed distributions arise in machine learning applications, see \cite{vladimirova_subweibull_2020,vladimirova_understanding_2019,gurbuzbalaban_heavy-tail_2021,zhu_beyond_2022,NEURIPS2020_abd1c782,NEURIPS2021_26901deb,nguyen_high_2023,jakovetic_nonlinear_2023}. This is especially relevant while training (deep) neural networks, and this phenomenon can also arise in decentralized learning set-ups \cite{gurbuzbalaban_heavy-tail_2022}.

\vspace{.2cm}

\noindent \textbf{Contributions.} Overall, this paper offers the following contributions:
\begin{enumerate}
    \item We study the convergence of~\cref{eq:bp-prototype}, and provide convergence results in mean and in high probability when the errors follow a sub-Weibull distribution. In particular, bounds are offered for the distance from the fixed point (if the operator is contractive) or for the \emph{cumulative fixed point residual} (if the operator is averaged). These error bounds hold for any iteration $\ell \in \N$ with a given, arbitrary probability. We also show that the high-probability bounds -- in the form $ \pr{\norm{x} \leq \epsilon(\delta)} \geq 1-\delta$ for the random variable $x$ -- scale with a factor $\sim \log(1/\delta)$, as opposed to a scaling $\sim 1/\delta$  that one would obtain via Markov's inequality.
    
    \item The framework we propose leverages a sub-Weibull model \cite{vladimirova_subweibull_2020,zhang_concentration_2021,kuchibhotla_moving_2022,wong_lasso_2020} for the norm of the additive errors. To the best of our knowledge, this is the first time that sub-Weibull models are employed in combination with operator-theoretic tools. By using sub-Weibulls, we are able to model a very broad class of random variables, fitting different practical applications.
    
    \item As mentioned above, the convergence results proposed in this paper hold with high probability, as opposed to the almost sure convergence results of \emph{e.g.} \cite{combettes_stochastic_2015}. We discuss the difference between the two viewpoints, and also further characterize the convergence of~\cref{eq:bp-prototype} in almost sure terms (under some additional assumptions, such as vanishing errors).
    
    \item The analysis provided also holds in online scenarios, in which the operator characterizing~\cref{eq:bp-prototype} changes over time. 
\end{enumerate}

\vspace{.2cm}

\noindent \textbf{Organization.} In \cref{sec:preliminaries} we review some preliminaries in operator theory, and introduce the sub-Weibull formalism. In \cref{sec:framework}, building on the motivating example of federated learning, we formalize and discuss the stochastic framework. In \cref{sec:mean-convergence,sec:high-probability} we present and discuss, respectively, mean and high probability convergence results derived in the proposed framework. Finally, \cref{sec:numerical} presents some illustrative numerical results.

%---------------------------------------------------------------------------------------------
\section{Preliminaries}\label{sec:preliminaries}

%--------------------------------------------------------------
\subsection{Operators}\label{subsec:operators}
Let $\Hs_i$, $i \in \{ 1, \ldots, n \}$ be (possibly infinite-dimensional) Hilbert spaces with inner product $\langle \cdot, \cdot \rangle_i$, induced norm $\norm{\cdot}_i$ and identity $\I_i : \Hs_i \to \Hs_i$. We consider the direct sum space $\Hs = \Hs_1 \oplus \cdots \oplus \Hs_n$, whose elements are $\x = (x_1, \ldots, x_n)$ with $x_i \in \Hs_i$ for all $i \in \{ 1, \ldots, n \}$. For $\x, \y \in \Hs$, we define the inner product  as $\langle \x, \y \rangle = \sum_{i = 1}^n \langle x_i, y_i \rangle_i$, and denote the induced norm and identity of $\Hs$ as $\norm{\cdot}$ and $\I : \Hs \to \Hs$, respectively. We consider operators $\T : \Hs \to \Hs$ defined as
\begin{equation}
    \T \x = (\T_1 \x, \ldots, \T_n \x)
\end{equation}
where $\T_i : \Hs \to \Hs_i$ for all $i \in \{ 1, \ldots, n \}$.

A central theme of the paper is to  compute fixed points of a given operator via iterative algorithms. To this end, in the following we introduce pertinent definitions and results. For a background on operator theory we refer to \emph{e.g.} \cite{cegielski_iterative_2012,bauschke_convex_2017}.

\begin{definition}[Fixed points]
Let $\T : \Hs \to \Hs : \x \mapsto \T \x = (\T_1 \x, \ldots, \T_n \x)$. The point $\bar{\x} \in \Hs$ is a \emph{fixed point} of $\T$ if $\bar{\x} = \T \bar{\x}$. We denote the \emph{fixed set} of $\T$ as $\fix(\T) = \{ \x \in \Hs \ | \x = \T \x \}$.
\end{definition}

\noindent Notice that a fixed point $\bar{\x} = (\bar{x}_1, \ldots, \bar{x}_n)$ by definition is such that $\bar{x}_i = \T_i \bar{\x}$ for all $i \in \{ 1, \ldots, n \}$. In the following, we review well-known definitions and properties of operators (see, e.g., \cite{bauschke_convex_2017}).

\begin{definition}[Non-expansive, contractive]\label{def:lipschitz-operators}
The operator  $\T$ is \emph{$\zeta$-Lipschitz continuous}, $\zeta \geq 0$, if:
\begin{equation}
	\norm{\T \x - \T \y} \leq \zeta \norm{\x - \y}, \quad \forall \x, \y \in \Hs.
\end{equation}
The operator is \emph{non-expansive} if $\zeta = 1$, and \emph{contractive} if $\zeta \in (0, 1)$. 
\end{definition}

\begin{definition}[Averaged]\label{def:averaged-operators}
The operator $\T$ is \emph{averaged} if and only if there exist $\alpha \in (0,1)$ and a non-expansive operator $\mathcal{R} : \Hs \to \Hs$ such that $\T = (1-\alpha) \I + \alpha \mathcal{R}$.
Equivalently, $\T$ is $\alpha$-averaged if
$$
	\norm{\T \x - \T \y}^2 \leq \norm{\x - \y}^2 - \frac{1 - \alpha}{\alpha} \norm{(\I - \T) \x - (\I - \T) \y}^2
$$
for any $\x, \y \in \Hs$.
\end{definition}

We discuss now the existence of fixed points of non-expansive operators.

\begin{lemma}[Browder's theorem]\label{lem:browder-theorem}
Let $\D \subset \Hs$ be a non-empty, convex, compact subset, and let $\T : \D \to \D$ be non-expansive, then $\fix(\T) \neq \emptyset$ \cite[Theorem~4.29]{bauschke_convex_2017}.
\end{lemma}

\begin{lemma}[Banach-Picard theorem]\label{lem:banach-picard}
Let $\T : \Hs \to \Hs$ be a contractive operator, then $\fix(\T)$ is a singleton \cite[Theorem~1.50]{bauschke_convex_2017}. The unique fixed point is the limit of the sequence generated by:
$$
	\x^{\ell+1} = \T \x^\ell, \quad \ell \in \N.
$$
\end{lemma}

If we want to compute a fixed point of the non-expansive operator $\T$, the following results can be used.

\begin{lemma}
Let $\T : \Hs \to \Hs$ be a non-expansive operator with $\fix(\T) \neq \emptyset$. Then for any $\alpha \in (0, 1)$, the $\alpha$-averaged operator $\T_\alpha = (1-\alpha) \I + \alpha \T$ is such that $\fix\left( \T_\alpha \right) = \fix(\T)$.
\end{lemma}
\begin{proof}
Let $\bar{\x} \in \fix(\T)$, by definition $\bar{\x} = \T \bar{\x}$. Therefore given $\alpha \in (0, 1)$ we have
$
	(1 - \alpha) \bar{\x} + \alpha \T \bar{\x} = \bar{\x} - \alpha \bar{\x} + \alpha \bar{\x} = \bar{\x}.
$
\end{proof}

\begin{lemma}[\km theorem]\label{lem:krasnoselskii-mann}
Let $\D \subset \Hs$ be a non-empty closed convex subset, and let $\T : \D \to \D$ be a non-expansive operator (notice that by \cref{lem:browder-theorem} $\fix(\T) \neq \emptyset$). Let $\alpha \in (0, 1)$, then the \km iteration
$$
	\x^{\ell+1} = (1-\alpha) \x^\ell + \alpha \T \x^\ell, \quad \ell \in \N,
$$
guarantees that $\norm{(\I - \T) \x^\ell} \to 0$ as $\ell \to \infty$ \cite[Theorem~5.15]{bauschke_convex_2017}.
\end{lemma}

In the remainder of the paper we will thus focus on studying the convergence of the Banach-Picard in two cases: i) when the operator is contractive, and ii) when it is averaged.

%--------------------------------------------------------------
\subsection{Probability}\label{subsec:background-probability}
In this section, we provide some definitions and results in probability theory that will be used in the paper to derive convergence results in high-probability. Throughout the paper, the underlying probability space will be $(\Omega, \mathcal{F}, \mathbb{P})$.
We start by introducing the definition of \textit{sub-Weibull random variables} \cite{wong_lasso_2020,vladimirova_subweibull_2020,zhang_concentration_2021,kuchibhotla_moving_2022}.

\begin{definition}[Sub-Weibull random variable]\label{def:sub-weibull}
A random variable $x$ is said to be sub-Weibull if there exist $\theta \geq 0$, $\nu > 0$ such that
$$
	\norm{x}_k = \ev{|x|^k}^{1/k} \leq \nu k^\theta, \quad \forall k \geq 1,
$$
and we denote it by $x \sim \sw(\theta, \nu)$.
\end{definition}

The following equivalence result relates the moments growth condition of \cref{def:sub-weibull} with a bound for the tail probability \cite[Lemma~5]{wong_lasso_2020}.

\begin{lemma}[Sub-Weibull tail probability]\label{lem:equivalent-sub-weibull}
Let $x \sim \sw(\theta, \nu)$, then the tail probability verifies the bound
\begin{equation}\label{eq:tail-probability}
	\pr{|x| \geq \epsilon} \leq 2 \exp\left( - \left( \frac{\epsilon}{c(\theta) \nu} \right)^{1 / \theta} \right), \quad \forall \,\, \epsilon > 0,
\end{equation}
where $c(\theta) := (2 e / \theta)^\theta$.
\end{lemma}

\cref{lem:equivalent-sub-weibull} shows that the tails of a sub-Weibull r.v. become heavier as the parameter $\theta$ grows larger. Moreover, setting $\theta = 1/2$ and $\theta = 1$ yields the class of sub-Gaussians and sub-exponential random variables, respectively; see, \emph{e.g.}, \cite{boucheron_concentration_2013,vershynin_high-dimensional_2018}. 
The following lemma shows how the tail probability equation~\cref{eq:tail-probability} can be used to give high probability bounds.

\begin{lemma}[High probability bound]\label{lem:high-probability-bound}
Let $x \sim \sw(\theta, \nu)$, then for any $\delta \in (0, 1)$, w.p. $1 - \delta$ we have the bound:
$$
    |x| \leq \log^\theta(2 / \delta) c(\theta) \ \nu.
$$
\end{lemma}
\begin{proof}
By \cref{lem:equivalent-sub-weibull} we have, for any $\epsilon > 0$:
$$
    \pr{|x| \geq \epsilon} \leq 2 \exp\left( - \left( {\epsilon} / {(c(\theta) \nu)} \right)^{1 / \theta} \right).
$$
Setting the right-hand side equal to $\delta$ and solving for $\epsilon$ we get $\epsilon = c(\theta) \nu \log^\theta(2 / \delta)$ which implies that, w.p. $1 - \delta$ we have $|x| \leq \epsilon = c(\theta) \nu \log^\theta(2 / \delta)$.
\end{proof}

We characterize now the properties of the sub-Weibull class of random variables.

\begin{lemma}[Inclusion]\label{lem:inclusion-sub-weibull}
Let $x \sim \sw(\theta, \nu)$, then $x \sim \sw(\theta', \nu')$ for any $\theta' \geq \theta$, $\nu' \geq \nu$.
\end{lemma}
\begin{proof}
By assumption we have $\norm{x}_k \leq \nu k^\theta$. Using the fact that $\nu k^\theta \leq \nu' k^{\theta'}$ (which holds since $k \geq 1$) yields the thesis; cf. \cite[Proposition~1]{vladimirova_subweibull_2020}.
\end{proof}

\begin{lemma}[Closure]\label{lem:closure-sub-weibull-class}
The class of sub-Weibull random variables is closed w.r.t. product by a scalar, sum, product, and exponentiation, according to the following rules.
\begin{enumerate}	
	\item \emph{Product by scalar:} let $x \sim \sw(\theta, \nu)$ and $a \in \R$, then $a x \sim \sw(\theta, |a| \nu)$;
	
	\item \emph{Sum:} let $x_i \sim \sw(\theta_i, \nu_i)$, $i \in \{ 1, 2 \}$, possibly dependent, then $x_1 + x_2 \sim \sw(\max\{ \theta_1, \theta_2 \}, \nu_1 + \nu_2)$;
	
	\item \emph{Product:} let $x_i \sim \sw(\theta_i, \nu_i)$, $i \in \{ 1, 2 \}$, and independent, then $x_1 x_2 \sim \sw(\theta_1 + \theta_2, \nu_1 \nu_2)$.
	
	\item \emph{Power:} let $x \sim \sw(\theta, \nu)$ and $a > 0$, then $x^a \sim \sw(a \theta, \nu^a \max\{ 1, a^{a \theta} \})$.
\end{enumerate}
\end{lemma}
\begin{proof} 
See \cref{proof:lem:closure-sub-weibull-class}.
\end{proof}

\smallskip

We conclude with some remarks.

\begin{remark}[Square of sub-Weibull]\label{rem:square-sub-weibull}
A consequence of 4) in \cref{lem:closure-sub-weibull-class} is that the square of $x \sim \sw(\theta, \nu)$ is itself a sub-Weibull, characterized by $x^2 \sim \sw(2 \theta, 4^\theta \nu^2)$. A particular case is the well known fact that the square of a sub-Gaussian ($\theta = 1/2$) is sub-exponential ($\theta = 1$), see \emph{e.g.} \cite[Lemma~2.7.6]{vershynin_high-dimensional_2018}.
\end{remark}

\begin{remark}[Mean of sub-Weibulls]\label{rem:mean-sub-weibull}
Notice that the definition of sub-Weibulls and their properties does not require that their mean be zero.
Moreover, if $x \sim \sw(\theta, \nu)$ and $x \geq 0$ almost surely, then $\ev{x} \leq \nu$, since $\norm{x}_1 = \ev{x}$.
\end{remark}

\begin{remark}[Bounded r.v.s]\label{rem:bounded-rv}
We can see that bounded r.v.s are sub-Weibull with $\theta = 0$; indeed, let $x$ be a r.v. such that a.s. $|x| \leq b$, then $\norm{x}_k \leq b = b k^0$, $k \geq 1$. This characterization is ``optimal'' in terms of $\theta$, since $\theta = 0$ corresponds to the lightest possible tail. However, it is sub-optimal in the other parameter, $\nu$, which does not reflect the overall distribution of $x$, only its maximum absolute value.
Alternatively, we know by \cite[p.~24]{vershynin_high-dimensional_2018} that the class of sub-Gaussian random variables includes that of bounded r.v.s; this implies that bounded r.v.s are sub-Weibull with $\theta = 1/2$.
\end{remark}

%---------------------------------------------------------------------------------------------
\section{Motivation and Framework Development}\label{sec:framework}
We start this section by discussing a motivating example in the context of \textit{federated learning} \cite{li_federated_2020,gafni_federated_2022}, around which we will formalize our stochastic operator framework.

%--------------------------------------------------------------
\subsection{Motivating example: federated learning}\label{subsec:motivating-ex}

\begin{figure}[!ht]
    \centering
    \includegraphics[scale=0.8]{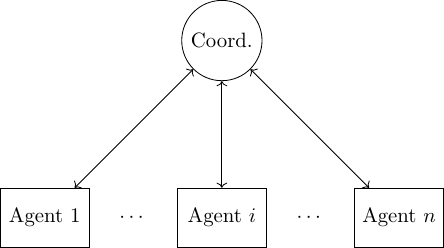}
    \caption{The architecture in federated learning.}
    \label{fig:fl-architecture}
\end{figure}

In \textit{federated learning}, a set of $n$ agents, aided by a coordinator, aim to cooperatively train a model without directly sharing the data they store \cite{li_federated_2020,gafni_federated_2022}. In order to keep the data private, the agents share the results of local training with the coordinator, which aggregates them into a more accurate model.
Formally, the goal is to solve the optimization problem
\begin{equation}\label{eq:federated-problem}
    \min_{x_i \in \R^q} \sum_{i = 1}^n f_i(x_i) \quad \text{s.t.} \ x_1 = x_2 = \ldots = x_n
\end{equation}
where $f_i : \R^q \to \R \cup \{ +\infty \}$ is the loss function of agent $i$, defined over the local dataset. Let $\{ d_i^h \}_{h = 1}^{m_i}$ be the $m_i$ data points stored by $i$, then usually the local loss is of the form
\begin{equation}\label{eq:local-loss}
    f_i(x_i) = \sum_{h = 1}^{m_i} \lambda(x_i; d_i^h)
\end{equation}
with $\lambda$ being a training loss (\textit{e.g.} quadratic or logistic).

In principle, \cref{eq:federated-problem} could be solved via projected gradient descent by applying, $i \in \{ 1, \ldots, n \}$:
\begin{equation}\label{eq:naive-federated-algorithm}
	x_i^{\ell+1} = \T_i x^\ell := \frac{1}{n} \sum_{j = 1}^n \left( x_j^\ell - \gamma \nabla f_j(x_j^\ell) \right), \quad \ell \in \N,
\end{equation}
where the agents compute local gradient descent steps, and the coordinator computes their average (the projection onto $x_1 = \ldots x_n$) and transmits it to the agents. In the following, using the notation defined in \cref{subsec:operators} we write~\cref{eq:naive-federated-algorithm} as $\x^{\ell+1} = \T \x^\ell$.
However, when solving learning problems there are several practical constraints that make implementing~\cref{eq:naive-federated-algorithm} unrealistic. In the following we discuss some of these practical challenges, and how they motivate the theoretical developments of subsequent sections. The section concludes by presenting a modified version of~\cref{eq:naive-federated-algorithm} that accounts for these challenges, and which can be interpreted as a stochastic operator, fitting in the framework of \cref{subsec:framework}.

%------------------------------
\subsubsection*{Challenge 1: additive errors}
Learning problems are often \textit{high dimensional}, both in the size of the unknown $x \in \R^q$ (especially when training neural networks), and in the size of the local data sets.
The size of the unknown poses a first challenge, since algorithm~\cref{eq:naive-federated-algorithm} requires sharing of $q$-dimensional vectors, and the larger $q$ the more expensive these communications become. In practice, then, \textit{quantization or compression} is applied before the agents communicate with the coordinator \cite{zhao_towards_2023}, lessening the burden at the cost of introducing some error.

A second challenge arises when computing the gradient of local losses~\cref{eq:local-loss}, since the larger $m_i$ is the more computationally expensive gradient computations are. To reduce the cost of gradient evaluations, a standard strategy in learning is to use \textit{stochastic gradients} \cite{yuan_federated_2020}, which approximate the true gradient using only a random subset of the local data points. But, similarly to communications reduction techniques, approximating local gradients introduces some error. In the following, we denote by $\hat{\nabla} f_i(x_i)$ a stochastic local gradient.

We are therefore interested in analyzing the modified version of~\cref{eq:naive-federated-algorithm} given by
$
	\x^{\ell+1} = \T \x^\ell + \e^\ell,
$
where $\e^\ell$ is a random vector modeling the errors introduced by the use of inexact communications and stochastic gradients.

\begin{remark}[Local training]
In practice, federated learning algorithm use \textit{local training}, that is, the agents perform multiple local gradient steps in between communication rounds. While this reduces the ratio of communications to gradient evaluations, it introduces \textit{client drift} \cite{karimireddy_scaffold_2020}; interpreting it as an additive error, one may then apply the results of the next sections to analyze the convergence of algorithms employing local training.
\end{remark}

%------------------------------
\subsubsection*{Challenge 2: asynchrony}
The agents cooperating in the learning procedure may be \textit{highly heterogeneous}, in that they have different computational resources \cite{gafni_federated_2022}. In~\cref{eq:naive-federated-algorithm} all the agents compute a local gradient at the same time -- however, heterogeneous agents will have different computation times. Therefore, requiring the agents to work synchronously implies that the faster agents will sit idle while waiting for the slower ones to conclude the local computations \cite{peng_coordinate_2016}.

One way to solve this problem is to allow \textit{asynchronous} activation of the agents, so that each agent is free to perform local training at its own pace.
More formally, let $\Is^\ell \subset \{ 1, \ldots, n \}$, $|\Is^\ell| \leq n$, be the subset of agents that concluded a local computation at time $\ell \in \N$. Then we are interested in algorithms in which the coordinator only aggregates information from $\Is^\ell$, instead of all agents.
\Cref{alg:federated-gradient} presents a modified version of~\cref{eq:naive-federated-algorithm} that allows for asynchronous activations (see \cref{app:derivation-fl-algorithm} for the derivation).

We also remark that asynchrony could be enforced by design as an additional measure to reduce the amount of communications required at each time.

%------------------------------
\subsubsection*{Challenge 3: online problems}
The problem~\cref{eq:federated-problem} discussed so far is static, in the sense that the data sets defining the local losses do not change over time. However, in many learning applications the agents may be continuously collecting new data and consequently modifying their loss.
This results in an \textit{online learning} problem, characterized by \cite{shalev-shwartz_online_2011,dallanese_convergence_2019,simonetto_time-varying_2020}
$$
    \min_{x_i \in \R^q} \sum_{i = 1}^n f_i^\ell(x_i) \quad \text{s.t.} \ x_1 = x_2 = \ldots = x_n.
$$
In this set-up the solution(s) to the problem at time $\ell \in \N$ may not coincide with those at time $\ell+1$, and our analysis of \cref{alg:federated-gradient}'s performance needs to account for this fact.

\begin{algorithm}[!ht]
\caption{Federated gradient descent.}
\label{alg:federated-gradient}
\begin{algorithmic}[1]
\Require{Initial conditions $z^0$, $x_i^0$, and step-size $\gamma$.}
\For{$\ell = 0, 1, 2, \ldots$}
	\Statex\hspace{1.5em}{\color{blue}// compressed communications}
	\State the coordinator compresses and transmits $z^\ell$ to the active agents $\Is^\ell$
	
	\For{$i \in \{ 1, \ldots, n \}$}
		\Statex\hspace{3em}{\color{blue}// asynchrony}
		\If{$i \in \Is^\ell$}
			\Statex\hspace{4.5em}{\color{blue}// inexact gradients}
			\State the agent performs a local (stochastic) gradient step
			$$
				x_i^{\ell+1} = z^\ell - \gamma \hat{\nabla} f_i(z^\ell)
			$$
			\Statex\hspace{4.5em}{\color{blue}// compressed communications}
			\State and compresses and transmits $x_i^{\ell+1}$ to the coordinator
		\Else
			\State otherwise sets $x_i^{\ell+1} = x_i^\ell$
		\EndIf
	\EndFor
	
	\State the coordinator computes
	$$
		z^{\ell+1} = \frac{1}{n} \left( \sum_{i \in \Is^\ell} x_i^{\ell+1} + \sum_{i \not\in \Is^\ell} x_i^\ell \right)
	$$
\EndFor
\end{algorithmic}
\end{algorithm}

%--------------------------------------------------------------
\subsection{Stochastic operator framework}\label{subsec:framework}

Motivated by the example discussed so far, we are now ready to formalize the stochastic operator framework of interest in this paper.
Consider an operator $\T \x = (\T_1 \x, \ldots, \T_n \x)$, in the following we analyze the convergence of the update, for $i \in \{ 1, \ldots, n \}$:
\begin{align}
	x_i^{\ell+1} &= \begin{cases}
		\T_i^{\ell+1} \x^\ell + e_i^\ell & \text{w.p.} \ p_i \\
		x_i^\ell & \text{w.p.} \ 1 - p_i
	\end{cases}\label{eq:stochastic-update} \\
	&= (1 - u_i^\ell) x_i^\ell + u_i^\ell \left( \T_i^{\ell+1} \x^\ell + e_i^\ell \right) \nonumber \\
	&=: \hat{\T}_i^{\ell+1} \x^\ell + u_i^\ell e_i^\ell \nonumber
\end{align}
where $\ell \in \N$ is the iteration index, $u_i^\ell$ are Bernoulli random variables that indicate whether a coordinate is updated or not at iteration $\ell$, and $\e^\ell = (e_1^\ell, \ldots, e_n^\ell)$ is a random vector of additive noise.
As discussed in \cref{subsec:motivating-ex}, in federated learning only the coordinates corresponding to active nodes (with $u_i^\ell$) are updated, and additive noise may be due to inexact communications or stochastic gradients (hence $e_i^\ell$). The operator $\T^\ell$ is allowed to change over time to account for online learning problems defined by streaming sources of data.

\smallskip

In the following we introduce and discuss some assumptions to formalize the framework. We remark that throughout the paper, the initial condition $\x^0 \in \Hs$ will be assumed to be deterministic.

\begin{assumption}[Stochastic framework 1]\label{as:stochastic-framework}
The following is assumed. 
\begin{itemize}

    \item[(i)] $\uv^\ell = (u_1^\ell, \ldots, u_n^\ell)$ is a $\{ 0, 1 \}^n$-valued random vector with $p_i := \pr{u_i = 1} > 0$, $i \in \{ 1, \ldots, n \}$.
    
    \item[(ii)] The additive error $\e^\ell \in \Hs$ is an $\Hs$-valued random vector such that $\ev{\norm{\e^\ell}^2} \leq \nu^2$.
    
    \item[(iii)] The random processes $\{ \uv^\ell \}_{\ell \in \N}$ and $\{ \e^\ell \}_{\ell \in \N}$ are i.i.d. and independent of each other.
\end{itemize}
\end{assumption}

We provide now two assumptions on the operator defining~\eqref{eq:stochastic-update}, which will be used to provide two different sets of results. Notice that these assumptions apply to the underlying \textit{deterministic} operator $\T^{\ell+1}$ only.

\begin{assumption}[Contractive set-up]\label{as:contractive}
Consider an operator $\T^\ell : \D \to \D$. The following is assumed. 
\begin{itemize}
	\item[(i)] The operator $\T^\ell : \Hs \to \Hs$ is $\zeta$-contractive for all $\ell \in \N$, and we denote by $\bar{\x}^\ell$ its unique fixed point.
	
	\item[(ii)] There exists $\sigma \geq 0$ such that the fixed points of consecutive operators have a bounded distance
	$$
		\norm{\bar{\x}^\ell - \bar{\x}^{\ell-1}} \leq \sigma, \quad \forall \ell \in \N.
	$$
\end{itemize}
\end{assumption}

\begin{assumption}[Averaged set-up]\label{as:averaged}
Consider an operator $\T^\ell : \D \to \D$. The following is assumed. 
\begin{itemize}
	\item[(i)] The operator $\T^\ell : \D \to \D$ is $\alpha$-averaged for all $\ell \in \N$, with $\D \subset \Hs$ being convex and compact.
	
	\item[(ii)] There exists $\sigma \geq 0$ such that given $\bar{\x}^\ell \in \fix(\T^\ell)$, $\ell \in \N$, then
	$$
		\inf_{\bar{\x}^{\ell-1} \in \fix(\T^{\ell-1})} \norm{\bar{\x}^\ell - \bar{\x}^{\ell-1}} \leq \sigma.
	$$
\end{itemize}
\end{assumption}

In the following we denote by $\diam(\D)$ the diameter of $\D$.

\smallskip

Before moving on to the convergence analysis, we discuss the framework and the assumptions that characterize it.

%------------------------------
\subsubsection{Update model}
We note that \cref{as:stochastic-framework}(i) does not require independence among the components of $\uv^\ell$. Independence is only assumed between $\uv^\ell$ and $\uv^h$, for any pair $\ell, h \in \N$, $\ell \neq h$.

Under \cref{as:stochastic-framework}(i) there may exist times during which \emph{none} of the coordinates are updated, because $u_i^\ell = 0$ for all $i \in \{ 1, \ldots, n \}$. This is different from the framework used in \emph{e.g.} \cite{combettes_stochastic_2015,combettes_stochastic_2019}, where at least one $u_i^\ell$ must always be different from zero. Accounting for such events is important in online scenarios, where the problem changes (\textit{e.g.} due to changes in the environment) even when updates are not performed.

%------------------------------
\subsubsection{Additive error model}
The error model characterized by \cref{as:stochastic-framework}(ii) implies that the variance of the additive error is bounded.
This is a common assumption in learning applications, where the error $\e^\ell$ quantifies the difference between the gradient and a stochastic approximation thereof \cite{ghadimi_stochastic_2013,yuan_federated_2020}:
$$
    \e^\ell = \nabla f(\x^\ell) - \hat{\nabla} f(\x^\ell).
$$
Additionally, some quantization or compression errors also verify this assumption. This is the case for uniform or random quantization \cite{reisizadeh_exact_2019}; see also \cite{zhao_towards_2023} for a comprehensive overview.

We remark that this model implies that errors are \textit{persistently present} over time. This sets the paper apart from previous works on stochastic operators in which the additive error is assumed to decay to zero (or, equivalently, that the error is multiplied by a decaying parameter) \cite{combettes_stochastic_2015,combettes_stochastic_2019,berinde_iterative_2007}.
Finally, we note that \cref{as:stochastic-framework,as:stochastic-framework-2} do not require independence of the error components at a fixed time $\ell$, or between the errors drawn at different times $\ell$ and $h$. Finally, the errors could also be \emph{biased}, that is, they could have mean different from zero.

%------------------------------
\subsubsection{Operators}
In \cref{as:averaged}, the operators are defined on a compact $\D$, which is verified when for example $\T^\ell$ contain a projection onto $\D$. On the one hand, this implies that Browder's theorem (cf. \cref{lem:browder-theorem}) holds, and hence that the operators have a non-empty fixed set. On the other, it allows us to guarantee that additive errors do not lead to divergence. This assumption is not required in \textit{e.g.} \cite{combettes_stochastic_2015} since the additive errors converge to zero asymptotically, and thus cannot lead to divergence.

In the context of convex optimization, we can relate \cref{as:contractive,as:averaged} to the properties of the cost function on which the operator $\T^\ell$ is defined. For example, in the federated learning set-up of \cref{subsec:motivating-ex} convex loss functions yield averaged operators, while strongly convex ones yield contractive operators \cite{davis_convergence_2016} (cf. also the discussion in \cref{app:derivation-fl-algorithm}).

Finally, a remark about the dynamic nature of the operators. As discussed in \cref{subsec:motivating-ex}, in many applications we may be interested in solving optimization problems that change over time, as new data come in \cite{dallanese_optimization_2020,simonetto_time-varying_2020}, which motivates our choice to analyze time-varying operators.
In this context, we can interpret~\cref{eq:stochastic-update} as an online algorithm, in which the output ($\x^\ell$) computed at time $\ell$ is used to warm-start the computation at time $\ell+1$. For this reason, it is necessary to provide bounds on the difference between consecutive operators (or equivalently, consecutive problems) to ensure that they are ``close enough'' for the warm-starting to provide an improvement in performance. This guarantee is provided by \cref{as:contractive}(ii) and \cref{as:averaged}(ii), which can be obtained for example bounding the rate of change of the operators.

%---------------------------------------------------------------------------------------------
\section{Mean Convergence Analysis}\label{sec:mean-convergence}
We start by characterizing in \cref{subsec:main-results-mean} the mean convergence of~\cref{eq:stochastic-update} for contractive and averaged operators. The results and their implications will be discussed in \cref{subsec:discussion-mean}, while \cref{subsec:corollaries-mean} will present some corollaries on the asymptotic convergence.

%--------------------------------------------------------------
\subsection{Main results}\label{subsec:main-results-mean}

\begin{proposition}[Mean -- contractive operators]\label{pr:contractive}
Let \cref{as:stochastic-framework,as:contractive} hold, and let $\{ \x^\ell \}_{\ell \in \N}$ be the sequence generated by~\cref{eq:stochastic-update}. Then we have the following bound, for any $\ell \in \N$
$$
	\ev{\norm{\x^\ell - \bar{\x}^\ell}} \leq \sqrt{\frac{\pmax}{\pmin}} \left( \chi^\ell \norm{\x^0 - \bar{\x}^0} + \frac{1 - \chi^\ell}{1 - \chi} (\chi \sigma + \nu) \right)
$$
where $\pmax := \max_i p_i$, $\pmin := \min_i p_i$, and
$
	\chi := \sqrt{1 - \pmin + \pmin \zeta^2} \in (0, 1).
$
\end{proposition}
\begin{proof}
See \cref{proof:pr:contractive}.
\end{proof}

\begin{proposition}[Mean -- averaged operators]\label{pr:averaged}
Let \cref{as:stochastic-framework,as:averaged} hold, and let $\{ \x^\ell \}_{\ell \in \N}$ be the sequence generated by~\cref{eq:stochastic-update}. Then we have the following bound, for any $\ell \in \N$
\begin{align*}
	&\ev{\frac{1}{\ell+1} \sum_{h = 0}^\ell \norm{(\I - \T^{h+1}) \x^h}^2} \leq \frac{\alpha}{\pmin (1 - \alpha)} \times \\ & \times \left( \frac{1}{\ell+1} \norm{\x^0 - \bar{\x}^0}^2 + \sigma^2 + \nu^2 + 2 \diam(\D) (\sigma + \nu) \right)
\end{align*}
\end{proposition}
\begin{proof}
See \cref{proof:pr:averaged}.
\end{proof}

%--------------------------------------------------------------
\subsection{Discussion}\label{subsec:discussion-mean}

%------------------------------
\subsubsection{Difference from deterministic convergence}\label{subsec:difference-from-deterministic}

In the following we discuss the difference between \cref{pr:contractive,pr:averaged} and the convergence of the static and deterministic update $\x^{\ell+1} = \T \x^\ell$, $\ell \in \N$.

\paragraph{Contractive case}
By contractiveness we know that $\x^{\ell+1} = \T \x^\ell$ converges linearly to $\fix(\T) = \{ \bar{\x} \}$, that is $\norm{\x^\ell - \bar{\x}} \leq \zeta^\ell \norm{\x^0 - \bar{\x}}$ \cite[Theorem~1.50]{bauschke_convex_2017}.
Comparing this bound with that of \cref{pr:contractive} we notice first of all that the introduction of random coordinate updates \textit{degrades the convergence rate} -- in mean -- from $\zeta$ to $\chi \geq \zeta$, in accordance with the results of \cite{combettes_stochastic_2019}.

Secondly, the presence of additive errors implies that we do not reach exact convergence to the fixed point, but rather to a neighborhood thereof. Moreover, at any given time $\ell \in \N$, the additive error may cause the overall stochastic operator to be \textit{expansive} $\norm{\x^{\ell+1} - \bar{\x}} \geq \norm{\x^\ell - \bar{\x}}$, but the underlying contractiveness of $\T$ ensure that this lead to inexact convergence rather than divergence.

\paragraph{Averaged case}
The convergence of~\cref{eq:stochastic-update} for averaged operators can be proved only in terms of the \emph{cumulative fixed point residual} $\frac{1}{\ell+1} \sum_{h = 0}^\ell \norm{(\I - \T^{h+1}) \x^h}^2$, due to averagedness being weaker than contractiveness.
On the other hand, for the deterministic update $\x^{\ell+1} = \T \x^\ell$ it is possible to prove that $\{ \norm{(\I - \T) \x^\ell} \}_{\ell \in \N}$ is a monotonically decreasing sequence that converges to zero \cite[Theorem~5.15]{bauschke_convex_2017}.
Similarly to the contractive case, this is no longer the case in the presence of additive errors, and the metric used to analyze convergence needs to account for the overall evolution of the fixed point residual.

We remark that the concept of cumulative fixed point residual is similar to that of \emph{regret} in convex optimization, and in particular to that of dynamic regret in online optimization \cite{mokhtari_online_2016}. Moreover, it includes as a particular case the different concept of regret based on the residual of the proximal gradient method proposed in \cite{hallak_regret_2021}.

%------------------------------
\subsubsection{Time-variability}
As mentioned in \cref{subsec:motivating-ex}, the stochastic framework we analyze allows for the operator to change over time, as this models \textit{online optimization algorithms} \cite{dallanese_optimization_2020,simonetto_time-varying_2020}.
But the time-variability of the operators, and hence of their fixed point(s), can be seen as a second source of additive errors besides $\e^\ell$. This error is always present, \textit{also when no update is performed}, to account for the fact that the environment changes at every iteration $\ell \in \N$. This is the case in online optimization when the problem depends on \textit{e.g.} measurements of a system, and the system will evolve even when no measurement is performed.

\paragraph{Path length}
In online optimization, a widely used concept is that of \emph{path-length}, that is, the cumulative distance between consecutive optimizers. This concept can be straightforwardly extended to the operator theoretical set-up, by defining it as the cumulative distance between consecutive fixed points:
$
	\sum_{h = 0}^{\ell-1} \norm{\bar{\x}^{h+1} - \bar{\x}^h} \leq \ell \sigma.
$
The path-length often appears in regret bounds, see \emph{e.g.} \cite{mokhtari_online_2016}, but notice that in the worst case it \emph{grows linearly with $\ell$}, and to carry out the convergence analysis the additional assumption that it grows sub-linearly is required.
What instead appears in our bound is a \emph{weighted} path-length
$$
	\sum_{h = 0}^{\ell-1} \chi^{\ell-h-1} \norm{\bar{\x}^{h+1} - \bar{\x}^h} \leq \frac{1 - \chi^\ell}{1 - \chi} \sigma
$$
which asymptotically reaches a fixed value, independent of $\ell$. This is due to the fact that we use the contractiveness of the operator, allowing to reach a tighter bound.

%------------------------------
\subsubsection{Convergence without random coordinate updates}
We remark that the results of \cref{pr:contractive,pr:averaged} can be adapted in a straightforward manner to a scenario where no random coordinate updates are performed. Indeed, if $p_i = 1$, $i \in \{ 1, \ldots, n \}$ then the two results provide bounds to the convergence of \emph{inexact operators} modeled by
$$
	\x^{\ell+1} = \T \x^\ell + \e^\ell, \quad \ell \in \N.
$$
This class of stochastic updates includes several optimization methods such as \emph{stochastic gradient descent} and \emph{$0$-th order methods} \cite{dixit_online_2019,liu_primer_2020,bottou2018optimization,duchi_optimal_2015,nesterov_random_2017,berahas_theoretical_2022}, which are widely used in machine learning applications.

\Cref{subsec:corollaries-hp} will instead provide some convergence results when only random coordinate updates are present.

%--------------------------------------------------------------
\subsection{Asymptotic convergence results}\label{subsec:corollaries-mean}
The mean convergence results of \cref{subsec:main-results-mean} can be further used to characterize the \textit{asymptotic, almost sure} convergence of~\eqref{eq:stochastic-update}, as proved in the following. The results are presented in the contractive case (cf. \cref{as:contractive}), but the same argument can be applied in the averaged case as well.

%------------------------------
\subsubsection{Convergence to a neighborhood}
We start by showing that the output of~\cref{eq:stochastic-update} asymptotically and almost surely converges to a bounded neighborhood of the fixed point trajectory $\{ \bar{\x}^\ell \}_{\ell \in \N}$.

\begin{corollary}[Asymptotic a.s. convergence to neighborhood]\label{cor:almost-sure-neighborhood}
Let \cref{as:stochastic-framework,as:contractive} hold, and let $\{ \x^\ell \}_{\ell \in \N}$ be the sequence generated by~\cref{eq:stochastic-update}.
Then it holds that:
$$
	\limsup_{\ell \to \infty} \norm{\x^\ell - \bar{\x}^\ell} \leq \sqrt{\frac{\pmax}{\pmin}} \frac{\chi \sigma + \nu}{1 - \chi} \qquad \text{a.s.}.
$$
\end{corollary}
\begin{proof}
See \cref{proof:cor:almost-sure-neighborhood}.
\end{proof}

%------------------------------
\subsubsection{Exact asymptotic convergence}
\cref{cor:almost-sure-neighborhood} shows that in general we can only prove asymptotic convergence to a neighborhood of the fixed point trajectory. The following result proves that a zero asymptotic error can be achieved under the assumptions that the operator is static and the additive error $\e^\ell$ vanishing.

\begin{corollary}[Exact a.s. convergence]\label{cor:almost-sure-exact}
Let \cref{as:stochastic-framework,as:contractive} hold, with the difference that
\begin{itemize}
	\item[(i)] the errors are vanishing, with $\ev{\norm{\e^\ell}^2} \leq (\nu^\ell)^2$ and $\lim_{\ell \to \infty} \nu^\ell = 0$;
	
	\item[(ii)] for all $\ell \in \N$, $\T^\ell = \T$; we denote by $\bar{\x}$ the unique fixed point of $\T$.
\end{itemize}
Let $\{ \x^\ell \}_{\ell \in \N}$ be the sequence generated by~\cref{eq:stochastic-update}, then
$$
	\limsup_{\ell \to \infty} \norm{\x^\ell - \bar{\x}} = 0 \qquad \text{a.s.}.
$$
\end{corollary}
\begin{proof}
See \cref{proof:cor:almost-sure-exact}.
\end{proof}

The result of \cref{cor:almost-sure-exact} is similar to \cite{combettes_stochastic_2015,combettes_stochastic_2019}, in which the additive errors are assumed to be vanishing (or to be multiplied by a vanishing parameter).

\smallskip

\cref{cor:almost-sure-neighborhood,cor:almost-sure-exact} leveraged Markov's inequality to characterize the almost sure behavior of~\cref{eq:stochastic-update} -- but these results hold only asymptotically, and do not characterize the transient.
The question then is: \textit{can we provide high probability bounds that hold for any $\ell \in \N$?}
A first observation is that almost sure convergence cannot be guaranteed during the transient, owing to the ever-present disturbance of the additive errors $\e^\ell$. The next section will instead provide transient bounds that hold with assigned probability.

%---------------------------------------------------------------------------------------------
\section{High Probability Convergence Analysis}\label{sec:high-probability}
In order to derive the high probability bounds for the transient, we introduce the following assumption, which models additive errors as sub-Weibull r.v.s.

\begin{assumption}[Stochastic framework 2]\label{as:stochastic-framework-2}
\cref{as:stochastic-framework}~(i)~and~(iii) hold, and moreover:
\begin{itemize}
    \item[(ii)] the additive error $\e^\ell \in \Hs$ is an $\Hs$-valued random vector such that its norm is sub-Weibull, that is $\norm{\e^\ell}^2 \sim \sw(\theta, \nu^2)$.
\end{itemize}
\end{assumption}

Sub-Weibull r.v.s can be used to model a broad range of additive errors that arise \textit{e.g.} in federated learning applications (cf. \cref{subsec:motivating-ex}).
For example, quantization and compression techniques that result in bounded additive errors can be modeled as sub-Weibulls as discussed in \cref{rem:bounded-rv}.

More importantly, a series of recent papers have highlighted how stochastic gradients present heavy tails in many learning applications \cite{gurbuzbalaban_heavy-tail_2021,zhu_beyond_2022,NEURIPS2020_abd1c782,NEURIPS2021_26901deb,nguyen_high_2023,jakovetic_nonlinear_2023}, especially involving the training of deep neural networks.
The same issues can therefore arise in a federated learning set-up, where each agent computes a local stochastic gradient \cite{gurbuzbalaban_heavy-tail_2022}.

%--------------------------------------------------------------
\subsection{Contractive set-up}\label{subsec:main-results-hp}
In the following we will make use of the scaled norm $\nnorm{\x}^2 := \sum_{i = 1}^n \frac{1}{p_i} \norm{x_i}_i^2$ introduced in \cref{proof:pr:contractive}.

\begin{proposition}[High prob. -- contractive operators]\label{pr:contractive-hp}
Let \cref{as:stochastic-framework-2,as:contractive} hold, and assume that the random coordinate update operators $\hat{\T}^\ell$ defined in~\cref{eq:stochastic-update} satisfy
\begin{equation}\label{eq:as-contractive-hp}
    \nnorm{\hat{\T}^\ell \x - \bar{\x}^\ell} \leq \eta(\ell) \nnorm{\x - \bar{\x}^\ell}
\end{equation}
with $\eta(\ell)$ i.i.d., and there exist $\gamma, \chi \in (0, 1)$ such that
\begin{equation}\label{eq:as-contractive-hp-2}
    \prod_{h = \ell_1}^{\ell_2} \eta(h) \sim \sw\left( \gamma^{\ell_2 - \ell_1 + 1}, \chi^{\ell_2 - \ell_1 + 1} \right).
\end{equation}

Let $\{ \x^\ell \}_{\ell \in \N}$ be the sequence generated by~\cref{eq:stochastic-update}, then for any $\ell \in \N$, given $\delta \in (0, 1)$, with probability $1 - \delta$ the following bound holds:
\begin{align*}
	\norm{\x^\ell - \bar{\x}^\ell} &\leq \sqrt{\frac{\pmax}{\pmin}} \log^{\theta + 1}(2 / \delta) c(\theta + 1) \times \\
    &\times \left( \chi^\ell \norm{\x^0 - \bar{\x}} + \frac{1 - \chi^{\ell}}{1 - \chi} (\chi \sigma + \nu) \right).
\end{align*}
\end{proposition}
\begin{proof}
See \cref{proof:pr:contractive-hp}.
\end{proof}

\smallskip

The following sections discuss the assumptions required in \cref{pr:contractive-hp}, as well as the implications of such convergence result.

%------------------------------
\subsubsection{\texorpdfstring{Assumption~\cref{eq:as-contractive-hp}}{Assumption~(10)}}\label{subsec:assumption-contractive}
In this section we discuss and motivate the use of assumption~(10) in \cref{pr:contractive-hp}.
We start by remarking that the assumption is stated in terms of the random coordinate updates operators $\hat{\T}^\ell$ defined in \cref{eq:stochastic-update}, and thus depends exclusively on $\{ u_i^\ell \}_{i = 1}^n$.
\begin{figure}[!ht]
    \centering
    \includegraphics[scale=0.6]{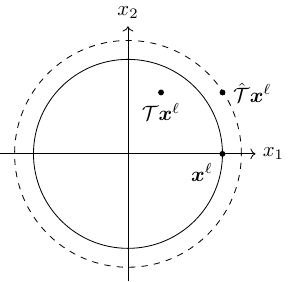}
    \caption{Random coordinate updates can lead to expansion.}
    \label{fig:expansion}
\end{figure}
In our setup, the underlying operators $\T^\ell$ are assumed to be contractive, but these random counterparts may not be, as exemplified in \cref{fig:expansion} where updating only the second coordinate leads to expansion. Assumption~(10) thus allows for the whole gamut of possibilities -- contraction, non-expansion, expansion -- since $\eta(\ell) \lesseqgtr 1$.

While an expansion is possible at any given time $\ell$, the results of \cref{sec:mean-convergence} motivate the second part of the assumption, \cref{eq:as-contractive-hp-2}.
First, the random operators $\hat{\T}^\ell$ are contractive in mean (w.r.t. the scaled norm $\nnorm{\cdot}$), since inspecting the proof of \cref{pr:contractive} in \cref{proof:pr:contractive} we see that (cf. \cref{eq:contractive-in-mean})
$$
	\ev{\nnorm{\hat{\T}^{\ell+1} \x^\ell - \bar{\x}^{\ell+1}}} \leq \chi \ev{\nnorm{\x^\ell - \bar{\x}^{\ell+1}}}
$$
with $\chi \in (0, 1)$. Thus by the discussion in \cref{rem:mean-sub-weibull} we expect the second sub-Weibull parameter of $\prod_{h = \ell_1}^{\ell_2} \eta(h)$ to be in $(0, 1)$.

Secondly, \cref{cor:almost-sure-exact} shows that the iterated application of a random coordinate update operator leads to almost sure convergence to the fixed point. This observation thus motivates the choice of a tail parameter for $\prod_{h = \ell_1}^{\ell_2} \eta(h)$ which decays as the number of iterations increases.

%------------------------------
\subsubsection{Interpretation of high probability convergence results}
The high probability bound of \cref{pr:contractive-hp} states that, w.p. $1 - \delta$, the error satisfies:
\begin{align*}
	\norm{\x^\ell - \bar{\x}^\ell} &\leq \sqrt{\frac{\pmax}{\pmin}} \log^{\theta + 1}(2 / \delta) c(\theta + 1) \times \\
    &\times \left( \chi^\ell \norm{\x^0 - \bar{\x}} + \frac{1 - \chi^{\ell}}{1 - \chi} (\chi \sigma + \nu) \right).
\end{align*}

First of all, we can see that the right-hand side is multiplied by $\log^{\theta + 1}(2 / \delta)$, which grows as $\delta$ decreases. This implies that the more confidence we want in the bound (which requires $\delta$ to be smaller), the looser the bound becomes. Intuitively, smaller $\delta$ requires that we enlarge the bound to include more trajectory realizations $\{ \x^\ell \}_{\ell \in \N}$.
Similarly, if $\theta$ grows larger, then the bound becomes looser. This is a consequence of the fact that the larger $\theta$ is, the heavier the tails of the additive noise are.

We can further observe that the high probability bound -- multiplicative factors notwithstanding -- shares the structure of the mean bound of \cref{pr:contractive}.
Indeed, both bounds have a first term that depends on the initial condition and decays to zero as $\ell \to \infty$, and a second term that bounds the asymptotic distance from the fixed point, and which depends on the additive error and the time-variability of the operators.

%------------------------------
\subsubsection{Sub-Weibull \emph{v.} Markov's inequality}
The fact that the dependence on $1 / \delta$ appears through its logarithm in the bound of \cref{pr:contractive-hp} is an important motivation for the use of the proposed sub-Weibull framework. This feature of the bounds ensures that the right-hand side grows relatively slowly when we ask for increasing confidence (that is, $\delta \to 0$).

Consider instead the following alternative high probability bound, which is based on Markov's inequality: with this approach, the dependence on the right-hand side is with $1 / \delta$ and not its logarithm. Although Markov's inequality holds for a more general class of random variables than sub-Weibull, this result makes it clear that using the sub-Weibull framework allows to derive sharper bounds, while still considering a number of relevant distributions as sub-cases. 

\begin{lemma}
Let \cref{as:stochastic-framework-2,as:contractive} hold, and let $\{ \x^\ell \}_{\ell \in \N}$ be the sequence generated by~\cref{eq:stochastic-update}. Then with probability $1 - \delta$, $\delta \in (0, 1)$, we have that
$$
	\norm{\x^\ell - \bar{\x}^\ell} \leq \frac{1}{\delta} \sqrt{\frac{\pmax}{\pmin}} \left( \chi^\ell \norm{\x^0 - \bar{\x}^0} + \frac{1 - \chi^\ell}{1 - \chi} (\chi \sigma + \nu) \right).
$$
\end{lemma}
\begin{proof}
By Markov's inequality we know that $\pr{\norm{\x^\ell - \bar{\x}^\ell} \geq \ev{\norm{\x^\ell - \bar{\x}^\ell} / \delta}} \leq \delta$. Using the bound of \cref{pr:contractive} then yields the thesis.
\end{proof}

%--------------------------------------------------------------
\subsection{Averaged set-up}
In this section, we provide a convergence analysis for averaged operators. The treatment of this case requires a somewhat more restrictive assumption than~\cref{eq:as-contractive-hp}, owing to the fact that averagedness is a weaker property than contractiveness. \Cref{subsec:assumption-averaged} will discuss this assumption in detail.

\begin{proposition}[High prob. -- averaged operators]\label{pr:averaged-hp}
Let \cref{as:stochastic-framework-2,as:averaged} hold, and assume that the random coordinate update operators $\hat{\T}^{\ell+1}$ defined in~\cref{eq:stochastic-update} satisfy
\begin{equation}\label{eq:as-averaged-hp}
\begin{split}
    &\nnorm{\hat{\T}^{\ell+1} \x - \bar{\x}^{\ell+1}}^2 \\ &\leq \nnorm{\x - \bar{\x}^{\ell+1}}^2 - \frac{1 - \alpha(\ell)}{\alpha(\ell)} \nnorm{(\I - \T^{\ell+1}) \x}^2
\end{split}
\end{equation}
where $\alpha(\ell) \in (0, 1)$ almost surely.

Let $\{ \x^\ell \}_{\ell \in \N}$ be the sequence generated by~\cref{eq:stochastic-update}, then for any $\ell \in \N$, given $\delta \in (0, 1)$, with probability $1 - \delta$ the following bound holds:
\begin{align*}
	&\frac{1}{\ell+1} \sum_{h = 0}^\ell \norm{(\I - \T^{h+1}) \x^h}^2 \leq \log^{\theta}(2/\delta) c(\theta) \frac{\pmax}{\pmin} \times \\ &\times \frac{\bar{\alpha}}{1 - \bar{\alpha}} \left( \frac{\norm{\x^0 - \bar{\x}^0}^2}{\ell+1} + \sigma^2 + \nu^2 + 2 \sqrt{\pmin} \diam(\D) (\sigma + \nu) \right),
\end{align*}
where $\bar{\alpha} = \max_{h \in \{ 0, \ldots, \ell\}} \alpha(h)$.
\end{proposition}
\begin{proof}
See \cref{proof:pr:averaged-hp}.
\end{proof}

%------------------------------
\subsubsection{\texorpdfstring{Assumption~\cref{eq:as-averaged-hp}}{Assumption~(12)}}\label{subsec:assumption-averaged}
As discussed in \cref{subsec:assumption-contractive}, the presence of random coordinate updates may cause expansion, and assumption~\cref{eq:as-averaged-hp} is introduced to exclude such behavior.
This is needed due to the weaker nature of averagedness. Indeed, we are currently not able to provide high probability bounds for the metric $\frac{1}{\ell+1} \sum_{h = 0}^\ell \norm{(\I - \T^{h+1}) \x^h}^2$ when expansion may occur.

Let us now discuss how assumption~\cref{eq:as-averaged-hp} can be guaranteed in practice.
If expansion may occur, then we can model the operators $\hat{\T}^\ell$ as being $\beta(\ell)$-conically averaged (in the scaled norm $\nnorm{\cdot}$) \cite{bartz_conical_2022}
\begin{align*}
    &\nnorm{\hat{\T}^{\ell+1} \x - \bar{\x}^{\ell+1}}^2 \\ &\leq \nnorm{\x - \bar{\x}^{\ell+1}}^2 - \frac{1 - \beta(\ell)}{\beta(\ell)} \nnorm{(\I - \T^{\ell+1}) \x}^2
\end{align*}
where $\beta(\ell) > 0$ almost surely. The case of $\beta(\ell) > 1$ then models expansion.
Inspecting \cref{fig:expansion}, it is clear that expansion -- when it occurs -- is bounded, which implies that $\beta(\ell) < \bar{\beta}$ a.s.

Therefore, following \cite[Proposition~2.9]{bartz_conical_2022}, we can guarantee averagedness by applying a \textit{relaxed version} of the operators $\hat{\T}^\ell$:
$$
    (1 - \gamma) \I + \gamma \hat{\T}^\ell
$$
for $\gamma \in (0, 1/\bar{\beta})$. This shows that modifying the operator being applied can indeed lead to assumption~\cref{eq:as-averaged-hp} being guaranteed.

%--------------------------------------------------------------
\subsection{Convergence without additive errors}\label{subsec:corollaries-hp}
We conclude this section with two high probability convergence results derived in the absence of additive errors ($\e^\ell = \0$ a.s.) and when the operator is static ($\T^\ell = \T$, $\ell \in \N$).
We remark that \cref{pr:contractive-hp,pr:averaged-hp} do hold in this scenario by setting $\nu = 0$ and $\sigma = 0$; however, the following results present more refined bounds.

\begin{proposition}[Without additive noise]\label{pr:no-additive-noise}
Let \cref{as:stochastic-framework-2,as:contractive} hold, with $\e^\ell = \0$ a.s. and $\T^\ell = \T$, for all $\ell \in \N$. Let $\{ \x^\ell \}_{\ell \in \N}$ be the trajectory generated by~\cref{eq:stochastic-update}.

Let $\epsilon \in (0, p]$ with $p = \prod_{i = 1}^n p_i$, then with probability $1 - \delta(\epsilon, \ell)$ we have:
$$
	\norm{\x^\ell - \bar{\x}} \leq \zeta^{\ell (p - \epsilon)} \norm{\x^0 - \bar{\x}}
$$
where $\fix(\T) = \{ \bar{\x} \}$ and
\begin{equation}\label{eq:sanov-theorem-parameters}
\begin{split}
	\delta(\epsilon, \ell) &= \exp\left( - \ell D(p - \epsilon || p) \right), \\
	D(p - \epsilon || p) &= (p - \epsilon) \log\left( 1 - \frac{\epsilon}{p} \right) + \\ &\quad + (1 - p + \epsilon) \log\left( 1 + \frac{\epsilon}{1 - p} \right).
\end{split}
\end{equation}
\end{proposition}
\begin{proof}
Setting $\sigma = 0$ and $\e^\ell = \0$ in~\cref{eq:intermediate-bound-4} we have
$$
	\norm{\x^\ell - \bar{\x}} \leq \zeta^{\beta(\ell)} \norm{\x^0 - \bar{\x}}
$$
where $\beta(\ell) \sim \mathcal{B}(\ell, p)$. Using Sanov's theorem \cite[Theorem~D.3]{mohri_foundations_2018} (in particular its symmetric form in \cite[eq.~(D.7)]{mohri_foundations_2018}) we know that
$$
	\pr{\beta(\ell) \leq \ell (p - \epsilon)} \leq \exp\left( - \ell D(p - \epsilon || p) \right)
$$
and, since $\pr{\zeta^{\beta(\ell)} \geq \zeta^{\ell (p - \epsilon)}} = \pr{\beta(\ell) \leq \ell (p - \epsilon)}$, the thesis follows.
\end{proof}

We observe that in \cref{pr:no-additive-noise} the convergence rate is characterized by $\zeta^{p - \epsilon}$, which is larger than the convergence rate $\zeta$ achieved in the deterministic case (cf. the discussion in \cref{subsec:difference-from-deterministic}).
We further notice that $\delta(\epsilon, \ell)$ is a decreasing function of $\ell$, with $\delta(\epsilon, \ell) \to 0$ as $\ell \to \infty$. This implies that asymptotically, the bound holds with probability $\lim_{\ell \to \infty} 1 - \delta(\epsilon, \ell) = 1$. As a consequence, \cref{pr:no-additive-noise} yields the known fact that \textit{asymptotic almost sure convergence is achieved} \cite{combettes_stochastic_2015,bianchi_coordinate_2016} -- but with the difference that a probabilistic bound is also provided for any $\ell \in \N$.

\smallskip

A similar result can also be derived for the averaged case.

\begin{proposition}[Without additive noise -- averaged]\label{pr:no-additive-noise-averaged}
Let \cref{as:stochastic-framework-2,as:averaged} hold, with $e_i = 0$ a.s. for all $i \in \{ 1, \ldots, n \}$; let $\{ \x^\ell \}_{\ell \in \N}$ be the trajectory generated by~\cref{eq:stochastic-update}.

Let $\epsilon \in (0, p]$ with $p = \prod_{i = 1}^n$, then with probability $1 - \delta(\epsilon, \ell+1)$ we have:
$$
	\norm{(\I - \T) \x^\ell}^2 \leq \frac{1}{(\ell + 1) (p - \epsilon)} \frac{\alpha}{1 - \alpha} \norm{\x^0 - \bar{\x}}^2
$$
where $\delta(\epsilon, \ell+1)$ and $D(p - \epsilon || p)$ are defined as in~\cref{eq:sanov-theorem-parameters}.
\end{proposition}
\begin{proof}
See \cref{proof:pr:no-additive-noise-averaged}.
\end{proof}

Recalling again the discussion in \cref{subsec:difference-from-deterministic}, in the deterministic case we know that $\{ \norm{(\I - \T) \x^\ell} \}_{\ell \in \N}$ is monotonically decreasing, and such that \cite{davis_convergence_2016}
$$
	\norm{(\I - \T) \x^\ell}^2 \leq \frac{1}{\ell+1} \frac{\alpha}{1 - \alpha} \norm{\x^0 - \bar{\x}}^2.
$$
Therefore we see that the introduction of random coordinate updates degrades the convergence rate, namely from $1 / (\ell+1)$ to the larger $1 / ((\ell+1) (p - \epsilon))$.

%---------------------------------------------------------------------------------------------
\section{Illustrative Numerical Results}\label{sec:numerical}
In this section we illustrate the accordance of the theoretical results provided in this paper with numerical results derived when applying the federated learning \cref{alg:federated-gradient} to solve a logistic regression problem.
In particular, we consider problem~\cref{eq:federated-problem},
$
    \min_{x_i \in \R^q} \sum_{i = 1}^n f_i(x_i) \quad \text{s.t.} \ x_1 = x_2 = \ldots = x_n,
$
where the local data $\{ (a_i^h, b_i^h) \in \R^{1 \times q} \times \{-1, 1\} \}_{h = 1}^{m_i}$ define the loss
$$
    f_i(x) = \sum_{h = 1}^{m_i} \log\left( 1 + \exp(- b_i^h a_i^h x) \right) + \frac{\epsilon}{2} \norm{x}^2.
$$
In our experiments we have $n = 10$ agents with $m_i = 150$ data-points each, and the problem size is $q = 10$. The data are randomly generated using the \texttt{make\_classification} utility of \texttt{sklearn} \cite{scikit-learn}.
The regularization weight is set to $\epsilon = 0.075$. Owing to the regularization, the problem is strongly convex, and thus the federated algorithm we apply is contractive (cf. \cref{app:derivation-fl-algorithm}).

In the following sections we apply \cref{alg:federated-gradient} to this problem, subject to the challenges of \textit{asynchrony} and \textit{additive errors}. The simulations are implemented using \texttt{tvopt} \cite{bastianello_tvopt_2021}, and all results are averaged over $100$ Monte Carlo iterations.

%--------------------------------------------------------------
\subsection{Asynchrony}
We start by considering \cref{alg:federated-gradient} in which each agent has a probability $p$ of updating at iteration $\ell \in \N$; therefore we choose $\{ u_i \}_{i = 1}^n$ to be i.i.d. Bernoulli of mean $p$.
\cref{tab:asynchrony} reports the empirical convergence rate attained for the different update probabilities. Such rate is estimated by computing the slope of the mean error curves $\{ \ev{\norm{\x^\ell - \bar{\x}}} \}_{\ell \in \N}$.
\begin{table}[!ht]
\begin{center}
\caption{Empirical convergence rate for different update probabilities.}
\label{tab:asynchrony}
\begin{tabular}{cc}
    Update probability      & Convergence rate         \\
    \hline
    $0.1$                   & $0.9985$                 \\
    $0.2$                   & $0.9972$                 \\
    $0.3$                   & $0.9958$                 \\
    $0.4$                   & $0.9945$                 \\
    $0.5$                   & $0.9931$                 \\
    $0.6$                   & $0.9917$                 \\
    $0.7$                   & $0.9904$                 \\
    $0.8$                   & $0.9890$                 \\
    $0.9$                   & $0.9876$                 \\
    $1.0$                   & $0.9862$                 \\
    \hline
\end{tabular}
\end{center}
\end{table}
The results show how the larger the probability of each agent performing an update, the smaller the convergence rate. This is in accordance with the results of \cref{pr:contractive}. We remark that the algorithm still converges to the optimal solution, even if at different rates.

%--------------------------------------------------------------
\subsection{Synthetic additive errors}
We evaluate now the effect of additive errors on the performance of \cref{alg:federated-gradient}, in a fully synchronous case ($p_i = 1$). In particular, each local update to $x_i^\ell$ (line~$5$ in the pseudo-code) is subject to an additive error $e_i^\ell$ whose components are drawn from the Weibull distribution
\footnote{Notice that usually the Weibull distribution is characterized by the CDF $1 - \exp(- (x / \nu)^c)$ \cite{rinne_weibull_2009}, and hence is indeed a sub-Weibull, with parameters $\theta = 1/c$ and $\nu$ (cf. \cite{vladimirova_subweibull_2020}).}.
With this choice, by \cref{lem:norm-sub-weibull} we know that the norm of $e_i^\ell$ is sub-Weibull and thus satisfies \cref{as:stochastic-framework}(ii).

\begin{lemma}[Norm of sub-Weibull vectors]\label{lem:norm-sub-weibull}
Let $e = [e_1, \ldots, e_q]^\top$ be a random vector in $\R^q$, $q < +\infty$, such that $e_i \sim \sw(\theta, \nu)$, $i \in \{ 1, \ldots, q \}$. Then the Euclidean norm of $e$ is sub-Weibull with
$$
    \norm{e}^2 \sim \sw(2 \theta, \max\{ 1, 2^{2 \theta} \} q \nu).
$$
\end{lemma}
\begin{proof}
We want to characterize $\norm{e}^2 = \sum_{i = 1}^q e_i^2$ as a sub-Weibull. By \cref{lem:closure-sub-weibull-class} we know that $e_i^2 \sim \sw(2 \theta, \max\{ 1, 2^{2\theta} \} \nu^2)$, and it follows that $\sum_{i = 1}^q e_i^2 \sim \sw(2 \theta, q \max\{ 1, 2^{2\theta} \} \nu^2)$.
\end{proof}

In \cref{tab:errors} we report the mean asymptotic error (computed as the maximum mean error in the last $4/5$ of the simulation) for different values of $\theta$ and $\nu$. Additionally, we report the accuracy of the trained model on the training and test sets, computed as the percentage of correctly classified data points. The test set consists of $150$ data points generated from the same distribution as the training points.
\begin{table}[!ht]
\begin{center}
\caption{Asymptotic error, training and test accuracies for different sub-Weibull distributions of the additive noise.}
\label{tab:errors}
\begin{tabular}{ccccc}
    \hline
    $(\theta, \nu)$     & As. err.      & Training acc. [\%]    & Test acc. [\%]    \\[0.35em]
    $(1/2, 0.01)$       & $1.305$       & $90.167$              & $86.653$          \\
    $(1/2, 0.1)$        & $17.842$      & $72.757$              & $76.727$          \\
    $(1, 0.01)$         & $1.482$       & $89.720$              & $85.653$          \\
    $(1, 0.1)$          & $20.406$      & $71.939$              & $76.627$          \\
    $(2, 0.01)$         & $3.162$       & $84.945$              & $84.527$          \\
    $(2, 0.1)$          & $43.330$      & $67.924$              & $72.060$          \\
    \hline
\end{tabular}
\end{center}
\end{table}
In accordance with the theoretical results (cf. \cref{pr:contractive-hp}), the heavier the tail (\textit{i.e.} the larger $\theta$) or the larger the scaling parameter $\nu$, the larger the error attained by the algorithm.
As expected, the accuracy also degrades the larger the errors are. However, it is interesting to notice that in moving from $\nu = 0.01$ to $\nu = 0.1$, the test accuracy becomes better than the training accuracy. This may be due to the fact that larger additive errors improve the generalization capabilities of the model.

%--------------------------------------------------------------
\subsection{Stochastic gradients}
We conclude this section by discussing the performance of \cref{alg:federated-gradient} when the agents use stochastic gradients during the local updates. In particular, the gradients are approximated by using a random subset of $B$ data points.
\begin{table}[!ht]
\begin{center}
\caption{Asymptotic error, training and test accuracies for different batch sizes of the stochastic gradients.}
\label{tab:stochastic-gradients}
\begin{tabular}{ccccc}
    \hline
    $B$        & As. err.      & Training acc. [\%]    & Test acc. [\%]    \\[0.35em]
    $1$        & $1.298$       & $92.251$              & $89.680$          \\
    $5$        & $1.117$       & $92.248$              & $89.753$          \\
    $10$       & $0.970$       & $92.289$              & $89.727$          \\
    $50$       & $0.468$       & $92.619$              & $88.673$          \\
    $100$      & $0.186$       & $92.803$              & $88.667$          \\
    Full grad. & $0$           & $92.733$              & $88.667$          \\
    \hline
\end{tabular}
\end{center}
\end{table}
We notice that the larger the batch size, the smaller the asymptotic error, since the additive error we introduce is smaller. This is also illustrated by the training accuracy which roughly increases with $B$. On the other hand, the test accuracy degrades as the batch size grows larger than $10$, signaling that the smaller the errors introduced by the stochastic gradients, the more the trained model is over-fitting the training data.

%---------------------------------------------------------------------------------------------
%---------------------------------------------------------------------------------------------
\appendices
\crefalias{section}{appsec}
\crefalias{subsection}{appsec}

%---------------------------------------------------------------------------------------------
\section{Proof of \texorpdfstring{\cref{lem:closure-sub-weibull-class}}{Lemma~8}}\label{proof:lem:closure-sub-weibull-class}
\emph{Proof of 1)} The result follows by $\norm{a x}_k = |a| \norm{x}_k \leq |a| \nu k^\theta$.

\emph{Proof of 2)} For completeness we report the proof provided in \cite[Proposition~3]{vladimirova_subweibull_2020}. Using the triangle inequality we write
\begin{align*}
	\norm{x_1 + x_2}_k &\leq \norm{x_1}_k + \norm{x_2}_k \overset{\text{(i)}}{\leq} \nu_1 k^{\theta_1} + \nu_2 k^{\theta_2} \\ &\overset{\text{(ii)}}{\leq} (\nu_1 + \nu_2) k^{\max\{ \theta_1, \theta_2 \}},
\end{align*}
where (i) holds by the assumption that $x_i$ are sub-Weibull, and (ii) holds since $k \geq 1$.

\emph{Proof of 3)} By definition of $\norm{\cdot}_k$ we can write
\begin{align*}
	\norm{x y}_k &\leq \ev{|x|^k |y|^k}^{1/k} \overset{\text{(i)}}{=} \ev{|x|^k}^{1/k} \ev{|y|^k}^{1/k} \\ &\overset{\text{(ii)}}{\leq} \nu_1 k^{\theta_1} \nu_2 k^{\theta_2} = \nu_1 \nu_2 k^{\theta_1 + \theta_2}
\end{align*}
where (i) holds by independence and (ii) by sub-Weibull assumption.

\emph{Proof of 4)} By definition of $\norm{\cdot}_k$ we have $\norm{x^a}_k = \ev{|x^a|^k}^{1/k} = \ev{|x|^{a k}}^{1/k}$. Now, we distinguish two cases: if $0 < a < 1$ then by Jensen's inequality we have
$$
    \ev{|x|^{a k}}^{1/k} \leq \left( \ev{|x|^k}^{1/k} \right)^a \leq (\nu k^\theta)^a;
$$
instead if $a \geq 1$ then we can write
$$
    \ev{|x|^{a k}}^{1/k} = \left( \ev{|x|^{a k}}^{1/a k} \right)^a \overset{\text{(i)}}{\leq} \left( \nu (a k)^\theta \right)^a = \nu^a a^{a \theta} k^\theta
$$
where (i) holds by the fact that $a k \geq 1$ and that $x$ is sub-Weibull. \qed

%---------------------------------------------------------------------------------------------
\section{\texorpdfstring{\cref{alg:federated-gradient}}{Algorithm~1}}\label{app:derivation-fl-algorithm}

%--------------------------------------------------------------
\subsection{Derivation of the algorithm}
Consider~\cref{eq:naive-federated-algorithm}, and redefine, $i \in \{ 1, \ldots, n \}$
$$
    x_i^{\ell+1} = z^\ell - \gamma \nabla f_i(z^\ell) = \mathcal{G}_i z^\ell
$$
where $z^{\ell} = \frac{1}{n} \sum_{i = 1}^n x_i^{\ell}$ represents the average computed by the coordinator.

Consider now the asynchronous set-up described in \cref{subsec:motivating-ex}, in which only the subset $\Is^\ell$ of the agents is active. In this case, we need to modify the local update as follows
$$
	x_i^{\ell+1} = \begin{cases}
		\mathcal{G}_i z^\ell & \text{if} \ i \in \Is^\ell \\
		x_i^\ell & \text{otherwise},
	\end{cases}
$$
so that only the active agents update their state. Since the coordinator receives new information only from the active agents, then we can modify its update as
$$
	z^{\ell+1} = \frac{1}{n} \left( \sum_{i \in \Is^\ell} x_i^{\ell+1} + \sum_{i \not\in \Is^\ell} x_i^\ell \right).
$$
With this update, the coordinator aggregates the new information received from $\Is^\ell$ with the most recent information received from the inactive agents (which may have been transmitted several iterations ago, \textit{e.g.} if $x_i^{\ell+1} = x_i^\ell = x_i^{\ell-1}$).

%--------------------------------------------------------------
\subsection{Interpretation as stochastic operator}
The goal now is to show that \cref{alg:federated-gradient} derived in the previous section can be interpreted as a stochastic operator that fits into the framework of the paper.

We start with the deterministic algorithm (all agents are always active). Using $z^\ell = \frac{1}{n} \sum_{i = 1}^n x_i^\ell$, the algorithm is described as
$$
	x_i^{\ell+1} = \mathcal{G}_i \left( \frac{1}{n} \sum_{j = 1}^n x_j^\ell \right).
$$
Letting $\x = (x_1, \ldots, x_n)$ the overall algorithm therefore is characterized by the update
$
	\x^{\ell+1} = \T \x^\ell = (\mathcal{G} \circ \mathcal{A}) \x^\ell
$
with $\circ$ denoting composition of operators, and where $\mathcal{G} \x = (\mathcal{G}_1 x_1, \ldots, \mathcal{G}_n x_n)$ and
$$
	\mathcal{A} \x = \left( \frac{1}{n} \sum_{i = 1}^n x_i, \ldots, \frac{1}{n} \sum_{i = 1}^n x_i \right).
$$

Let us now turn to the asynchronous set-up in which only the active agents $\Is^\ell$ perform an update. Letting $u_i^\ell$ be a Bernoulli r.v. which is $1$ if agent $i$ updates at time $\ell$, then we can write
$$
	x_i^{\ell+1} = (1 - u_i^\ell) x_i^\ell + \T_i \x^\ell,
$$
which fits exactly into the framework of \cref{subsec:framework}.

We conclude this section by discussing the properties of $\T$ as derived from the properties of the problem
$$
	\min_{x_i \in \R^q} \sum_{i = 1}^n f_i(x_i) \quad \text{s.t.} \ x_1 = x_2 = \ldots = x_n.
$$
We remark first that $\mathcal{A}$ is non-expansive, and we need to characterize the properties of $\mathcal{G}$.
Assume that the local costs have $\lmax$-Lipschitz continuous gradients, and that they are $\lmin$-strongly convex, where we allow $\lmin = 0$ to signify that the costs are convex. Then we have the following cases:
\begin{itemize}
	\item convex costs ($\lmin = 0$): if $\gamma < 2 / \lmax$ then $\mathcal{G}$ is averaged, and hence so is $\T$ \cite[section~3.3]{davis_convergence_2016};
	
	\item strongly convex costs ($\lmin > 0$): if $\gamma < 2 / \lmax$ then $\mathcal{G}$ is contractive and by \cite[Lemma~4.11]{bauschke_firmly_2012} so is $\mathcal{T}$.
\end{itemize}
Finally, notice that the fixed point(s) of $\T$ do not coincide with solutions of the optimization problem, but rather given $\bar{\x} \in \fix(\T)$ then
$
	\bar{z} = \frac{1}{n} \sum_{i = 1}^n \bar{x}_i
$
is a solution to the problem.

%---------------------------------------------------------------------------------------------
\section{Proofs of \texorpdfstring{\Cref{sec:mean-convergence}}{Section~IV}}\label{proof:sec:mean-convergence}

%--------------------------------------------------------------
\subsection{Proof of \texorpdfstring{\cref{pr:contractive}}{Proposition~1}}\label{proof:pr:contractive}
Similarly to \textit{e.g.} \cite{combettes_stochastic_2015,bianchi_coordinate_2016}, the first step is to define the norm
$$
	\nnorm{\x} := \sqrt{\sum_{i = 1}^n \frac{1}{p_i} \norm{x_i}_i^2}
$$
for which, letting $\pmax := \max_i p_i$ and $\pmin = \min_i p_i$, it holds that
\begin{equation}\label{eq:norm-equivalence}
	\pmin \nnorm{\x}^2 \leq \norm{\x}^2 \leq \pmax \nnorm{\x}^2.
\end{equation}
By the triangle inequality and the fact that $\nnorm{(u_1^\ell e_1^\ell, \ldots, u_n^\ell e_n^\ell)} \leq \nnorm{\e^\ell}$ we have
$$
	\nnorm{\x^{\ell+1} - \bar{\x}^{\ell+1}} \leq \nnorm{\hat{\T}^{\ell+1} \x^\ell - \bar{\x}^{\ell+1}} + \nnorm{\e^\ell}
$$
where $\hat{\T}^{\ell+1}$ is defined in~\cref{eq:stochastic-update}. Taking the expected value, by~\eqref{eq:norm-equivalence} we have $\ev{\nnorm{\e^\ell}} \leq \ev{\norm{\e^\ell}} / \sqrt{\pmin}$, and by \cref{as:stochastic-framework}(ii) and Jensen's inequality we have $\ev{\norm{\e^\ell}} \leq \nu$, hence
$
	\ev{\nnorm{\e^\ell}} \leq \nu / \sqrt{\pmin}.
$
We focus now on the first term. By the law of total expectation and Jensen's inequality we have
\begin{align*}
	&\ev{\nnorm{\hat{\T}^{\ell+1} \x^\ell - \bar{\x}^{\ell+1}}} = \ev{\ev[\ell]{\nnorm{\hat{\T}^{\ell+1} \x^\ell - \bar{\x}^{\ell+1}}}} \\
	&\hspace{3cm} \leq \ev{\sqrt{\ev[\ell]{\nnorm{\hat{\T}^{\ell+1} \x^\ell - \bar{\x}^{\ell+1}}^2}}}
\end{align*}
where $\ev[\ell]{\cdot}$ denotes the expectation conditioned on $\x^\ell$. By linearity of the expected value we can write
\begin{align}
	&\ev[\ell]{\nnorm{\hat{\T}^{\ell+1} \x^\ell - \bar{\x}^{\ell+1}}^2} = \nonumber \\
	&= \sum_{i = 1}^n \ev[\ell]{\frac{1}{p_i} \norm{\hat{\T}_i^{\ell+1} \x^\ell - \bar{x}_i^{\ell+1}}_i^2} \nonumber \\
	&\overset{(a)}{=} \sum_{i = 1}^n \ev[\ell]{\frac{u_i^\ell}{p_i} \norm{\T_i^{\ell+1} \x^\ell - \bar{x}_i^{\ell+1}}_i^2 + \frac{1 - u_i^\ell}{p_i} \norm{x_i^\ell - \bar{x}_i^{\ell+1}}_i^2} \nonumber \\
	&\overset{(b)}{=} \norm{\T^{\ell+1} \x^\ell - \bar{\x}^{\ell+1}}^2 + \sum_{i = 1}^n \frac{1 - p_i}{p_i} \norm{x_i^\ell - \bar{x}_i^{\ell+1}}_i^2 \label{eq:intermediate-bound} \\
	&\overset{(c)}{\leq} \zeta^2 \norm{\x^\ell - \bar{\x}^{\ell+1}}^2 + \sum_{i = 1}^n \frac{1 - p_i}{p_i} \norm{x_i^\ell - \bar{x}_i^{\ell+1}}_i^2 \nonumber \\
	&= \sum_{i = 1}^n \frac{1}{p_i} (1 - p_i + p_i \zeta^2) \norm{x_i^\ell - \bar{x}_i^{\ell+1}}_i^2 \nonumber \\
    &\overset{(d)}{\leq} \chi^2 \nnorm{\x^\ell - \bar{\x}^{\ell+1}}^2 \label{eq:contractive-in-mean}
\end{align}
where ($a$) holds by definition of $\hat{\T}^{\ell+1}$, ($b$) by the fact that $u_i^\ell \sim Ber(p_i)$, ($c$) holds by the contractiveness in \cref{as:contractive}(i), and ($d$) by defining $\chi = \sqrt{\max_i 1 - p_i + p_i \zeta^2} = \sqrt{1 - \pmin + \pmin \zeta^2}$.

Putting this bound together with that for $\ev{\nnorm{\e^\ell}}$:
\begin{align}
	\ev{\nnorm{\x^{\ell+1} - \bar{\x}^{\ell+1}}} &\leq \chi \ev{\nnorm{\x^\ell - \bar{\x}^{\ell+1}}} + \nu / \sqrt{\pmin} \nonumber \\
	&\hspace{-1.5cm}\leq \chi \ev{\nnorm{\x^\ell - \bar{\x}^\ell}} + \chi \sigma + \nu / \sqrt{\pmin} \label{eq:intermediate-bound-3}
\end{align}
where the last inequality holds by triangle inequality and \cref{as:contractive}(ii). Iterating and using the geometric sum
$$
	\ev{\nnorm{\x^\ell - \bar{\x}^\ell}} \leq \chi^\ell \nnorm{\x^0 - \bar{\x}^0} + \frac{1 - \chi^\ell}{1 - \chi} (\chi \sigma + \frac{\nu}{\sqrt{\pmin}}),
$$
and the thesis follows by~\cref{eq:norm-equivalence}. \qed

%--------------------------------------------------------------
\subsection{Proof of \texorpdfstring{\cref{pr:averaged}}{Proposition~2}}\label{proof:pr:averaged}
Let $\llangle \cdot \rrangle$ be the inner product that induces $\nnorm{\cdot}$, then we can write
\begin{align*}
	&\nnorm{\x^{\ell+1} - \bar{\x}^{\ell+1}}^2 = \\
	&=\nnorm{\hat{\T}^{\ell+1} \x^\ell - \bar{\x}^{\ell+1}}^2 + \nnorm{(u_1^\ell e_1^\ell, \ldots, u_n^\ell e_n^\ell)}^2 + \\ &+ 2 \llangle \hat{\T}^{\ell+1} \x^\ell - \bar{\x}^{\ell+1}, (u_1^\ell e_1^\ell, \ldots, u_n^\ell e_n^\ell) \rrangle \\
	&\leq \nnorm{\hat{\T}^{\ell+1} \x^\ell - \bar{\x}^{\ell+1}}^2 + \nnorm{\e^\ell}^2 + 2 \diam(\D) \sqrt{\pmin} \nnorm{\e^\ell}
\end{align*}
where the last inequality follows by $\nnorm{(u_1^\ell e_1^\ell, \ldots, u_n^\ell e_n^\ell)} \leq \nnorm{\e^\ell}$ and by Cauchy-Schwarz inequality.
Using the law of total expectation we can write
$
	\ev{\nnorm{\hat{\T}^{\ell+1} \x^\ell - \bar{\x}^{\ell+1}}^2} = \ev{\ev[\ell]{\nnorm{\hat{\T}^{\ell+1} \x^\ell - \bar{\x}^{\ell+1}}^2}}
$
where recall that $\ev[\ell]{\cdot}$ is conditioned on $\x^\ell$. Following the steps leading to~\cref{eq:intermediate-bound} we have
\begin{align*}
	&\ev[\ell]{\nnorm{\hat{\T}^{\ell+1} \x^\ell - \bar{\x}^{\ell+1}}^2} = \\
	&= \norm{\T^{\ell+1} \x^\ell - \bar{\x}^{\ell+1}}^2 + \sum_{i = 1}^n \frac{1 - p_i}{p_i} \norm{x_i^\ell - \bar{x}_i^{\ell+1}}_i^2 \\
	&\overset{(a)}{\leq} \norm{\x^\ell - \bar{\x}^{\ell+1}}^2 - \frac{1-\alpha}{\alpha} \norm{(\I - \T^{\ell+1}) \x^\ell}^2 + \\ &\hspace{1.5cm} + \sum_{i = 1}^n \frac{1 - p_i}{p_i} \norm{x_i^\ell - \bar{x}_i^{\ell+1}}_i^2 \\
	&= \nnorm{\x^\ell - \bar{\x}^{\ell+1}}^2 - \frac{1-\alpha}{\alpha} \norm{(\I - \T^{\ell+1}) \x^\ell}^2
\end{align*}
where ($a$) follows by averagedness in \cref{as:averaged}(i). Let $\bar{\x}^\ell \in \fix(\T^\ell)$, then
\begin{align*}
	&\nnorm{\x^\ell - \bar{\x}^{\ell+1}}^2 = \nnorm{\x^\ell - \bar{\x}^\ell}^2 + \nnorm{\bar{\x}^\ell - \bar{\x}^{\ell+1}}^2 + \\ &\hspace{2cm} + 2 \llangle \x^\ell - \bar{\x}^\ell, \bar{\x}^\ell - \bar{\x}^{\ell+1} \rrangle \\
	&\leq \nnorm{\x^\ell - \bar{\x}^\ell}^2 + \sigma^2 / \pmin + 2 \diam(\D) \sigma / \sqrt{\pmin}
\end{align*}
where we used \cref{as:averaged}(ii), \cref{eq:norm-equivalence}, and the fact that $\D$ is bounded.
Putting all these results together yields
\begin{align}
	&\ev{\nnorm{\x^{\ell+1} - \bar{\x}^{\ell+1}}^2} \leq \ev{\nnorm{\x^\ell - \bar{\x}^\ell}^2} + \sigma^2 / \pmin + \nonumber \\
	&+ 2 \diam(\D) \sigma / \sqrt{\pmin} + \ev{\nnorm{\e^\ell}^2 + 2 \diam(\D) \nnorm{\e^\ell}} + \nonumber \\
	&- \frac{1-\alpha}{\alpha} \norm{(\I - \T^{\ell+1}) \x^\ell}^2 \nonumber \\
	&\leq \ev{\nnorm{\x^\ell - \bar{\x}^\ell}^2} + d - \frac{1-\alpha}{\alpha} \norm{(\I - \T^{\ell+1}) \x^\ell}^2 \label{eq:intermediate-bound-2}
\end{align}
where the bound
$$
	d := \frac{1}{\pmin} \left( \sigma^2 + \nu^2 \right) + \frac{2 \diam(\D)}{\sqrt{\pmin}} \left( \sigma + \nu \right)
$$
was derived using \cref{as:stochastic-framework}~(ii) and Jensen's inequality.

Reordering~\cref{eq:intermediate-bound-2} and averaging over time yields
\begin{align*}
	&\frac{1}{\ell+1} \sum_{h = 0}^\ell \ev{\norm{(\I - \T^{h+1}) \x^h}^2} \leq \\
	&\qquad \leq \frac{\alpha}{1 - \alpha} \left( \frac{1}{\ell+1} \nnorm{\x^0 - \bar{\x}^0}^2 + d \right)
\end{align*}
where we used the telescopic sum and removed the negative term $- \ev{\nnorm{\x^{\ell+1} - \bar{\x}^{\ell+1}}^2}$. The thesis follows by~\cref{eq:norm-equivalence} and $\sqrt{\pmin} \leq 1$. \qed

%--------------------------------------------------------------
\subsection{Proof of \texorpdfstring{\cref{cor:almost-sure-neighborhood}}{Corollary~1}}\label{proof:cor:almost-sure-neighborhood}
By \cref{pr:contractive} we know that
$$
	\ev{\norm{\x^\ell - \bar{\x}^\ell}} \leq \sqrt{\frac{\pmax}{\pmin}} \left( \chi^\ell \norm{\x^0 - \bar{\x}^0} + \frac{1 - \chi^\ell}{1 - \chi} (\chi \sigma + \nu) \right)
$$
and, defining the random variable
$$
    y^\ell := \max\left\{ 0, \sqrt{\frac{\pmin}{\pmax}} \norm{\x^\ell - \bar{\x}^\ell} - \frac{1 - \chi^\ell}{1 - \chi} (\chi \sigma + \nu) \right\},
$$
this fact implies $\ev{y^\ell} \leq \chi^\ell \norm{\x^0 - \bar{\x}^0}$.

By Markov's inequality then we have that, for any $\epsilon > 0$:
$
    \pr{y^\ell \geq \epsilon} \leq \ev{y^\ell} / \epsilon \leq \chi^\ell \norm{\x^0 - \bar{\x}^0} / \epsilon,
$
and summing over time yields
$$
    \sum_{\ell = 0}^\infty \pr{y^\ell \geq \epsilon} \leq \frac{1}{1 - \chi} \frac{\norm{\x^0 - \bar{\x}^0}}{\epsilon} < \infty.
$$
By the Borel-Cantelli lemma, this fact implies that almost surely
$
    \limsup_{\ell \to \infty} y^\ell \leq \epsilon,
$
and, since the inequality holds for any $\epsilon > 0$ the thesis is proved. \qed

%--------------------------------------------------------------
\subsection{Proof of \texorpdfstring{\cref{cor:almost-sure-exact}}{Corollary~2}}\label{proof:cor:almost-sure-exact}
By assumption~(ii) the operator is static, which implies that $\sigma = 0$; hereafter $\bar{\x}$ denotes the unique fixed point of $\T$.

Following the same derivation leading to~\eqref{eq:intermediate-bound-3} yields
$$
	\ev{\nnorm{\x^{\ell+1} - \bar{\x}}} \leq \chi \ev{\nnorm{\x^\ell - \bar{\x}}} + \nu^\ell / \sqrt{\pmin}
$$
with the difference that now the right-most term is a function of $\ell$ as well. Iterating and using~\cref{eq:norm-equivalence} we have
$$
	\ev{\norm{\x^\ell - \bar{\x}}} \leq \sqrt{\frac{\pmax}{\pmin}} \left( \chi^\ell \norm{\x^0 - \bar{\x}} + \sum_{h = 0}^{\ell-1} \chi^{\ell - h - 1} \nu^h \right).
$$
Similarly to \cref{cor:almost-sure-neighborhood}, by Markov's inequality we have
$$
	\lim_{\ell \to \infty} \norm{\x^\ell - \bar{\x}} \leq \sqrt{\frac{\pmax}{\pmin}} \lim_{\ell \to \infty} \sum_{h = 0}^{\ell-1} \chi^{\ell - h - 1} \nu^h = 0
$$
where the right-hand side is equal to zero since $\{ \nu^\ell \}_{\ell \in \N}$ is summable and we can apply \cite[Lemma~3.1(a)]{sundharram_distributed_2010} because $\chi \in (0, 1)$. \qed

%---------------------------------------------------------------------------------------------
\section{Proofs of \texorpdfstring{\Cref{sec:high-probability}}{Section~V}}\label{proof:sec:high-probability}

%--------------------------------------------------------------
\subsection{Proof of \texorpdfstring{\cref{pr:contractive-hp}}{Proposition~3}}\label{proof:pr:contractive-hp}
We start by deriving the following chain of inequalitites:
\begin{align*}
    &\nnorm{\x^{\ell+1} - \bar{\x}^{\ell+1}} \overset{(i)
    }{\leq} \nnorm{\hat{\T}^{\ell+1} \x^\ell - \bar{\x}^{\ell+1}} + \nnorm{\e^\ell} \\
    &\quad \overset{(ii)}{\leq} \eta(\ell) \nnorm{\x^\ell - \bar{\x}^{\ell+1}} + \nnorm{\e^\ell} \\
    &\quad \overset{(iii)}{\leq} \eta(\ell) \nnorm{\x^\ell - \bar{\x}^\ell} + \eta(\ell) \nnorm{\bar{\x}^{\ell+1} - \bar{\x}^\ell} + \nnorm{\e^\ell} \\
    &\quad \overset{(iv)}{\leq} \eta(\ell) \nnorm{\x^\ell - \bar{\x}^\ell} + \eta(\ell) \frac{\sigma}{\sqrt{\pmin}} + \frac{\norm{\e^\ell}}{\sqrt{\pmin}}
\end{align*}
where (i) follows by the triangle inequality and the fact that $\nnorm{(\ldots, u_i^\ell e_i^\ell, \ldots)} \leq \nnorm{\e^\ell}$; (ii) by assumption~\cref{eq:as-contractive-hp}; (iii) by triangle inequality, and (iv) by \cref{eq:norm-equivalence} and \cref{as:contractive}(ii).

Iterating this inequality and using \cref{eq:norm-equivalence} again, we derive
\begin{equation}\label{eq:intermediate-bound-4}
\begin{split}
    &\nnorm{\x^{\ell+1} - \bar{\x}^{\ell+1}} \leq \frac{1}{\sqrt{\pmin}} \Bigg( \prod_{h = 0}^\ell \eta(h) \norm{\x^0 - \bar{\x}^0} + \\
    &\qquad + \sigma \sum_{h = 0}^\ell \prod_{j = h}^\ell \eta(j) + \sum_{h = 0}^\ell \norm{\e^h} \prod_{j = h+1}^\ell \eta(j) \Bigg).
\end{split}
\end{equation}

The goal now is to show that the right hand side of \cref{eq:intermediate-bound-4} is sub-Weibull.
First of all, by \cref{eq:as-contractive-hp}, \cref{lem:closure-sub-weibull-class}, and simplifying, we get that
$$
    \sum_{h = 0}^\ell \prod_{j = h}^\ell \eta(j) \sim \sw\left( \gamma, \chi \frac{1 - \chi^{\ell+1}}{1 - \chi} \right).
$$
Similarly, using the fact that $\norm{\e^\ell}^2 \sim \sw(\theta, \nu)$ and hence $\norm{\e^\ell} \sim \sw(\theta/2, \nu)$ by \cref{lem:closure-sub-weibull-class}, we get
$$
    \sum_{h = 0}^\ell \norm{\e^h} \prod_{j = h+1}^\ell \eta(j) \sim \sw\left( \theta+1, \nu \frac{1 - \chi^{\ell+1}}{1 - \chi} \right).
$$

Combining these results, and using \cref{lem:closure-sub-weibull-class} again, yields
\begin{align*}
	\cref{eq:intermediate-bound-4} \sim & \frac{1}{\sqrt{\pmin}} \sw\Big( \theta+1, \\ &\chi^{\ell+1} \norm{\x^0 - \bar{\x}^0} + \frac{1 - \chi^{\ell+1}}{1 - \chi} (\chi \sigma + \nu) \Big).
\end{align*}
Using \cref{eq:norm-equivalence} and \cref{lem:high-probability-bound} then yields the thesis. \qed

%--------------------------------------------------------------
\subsection{Proof of \texorpdfstring{\cref{pr:averaged-hp}}{Proposition~4}}\label{proof:pr:averaged-hp}
Using the fact that $\nnorm{(\ldots, u_i^\ell e_i^\ell, \ldots)} \leq \nnorm{\e^\ell}$, and Cauchy-Schwarz inequality we can write
\begin{align*}
	&\nnorm{\x^{\ell+1} - \bar{\x}^{\ell+1}}^2 \leq \\ &\leq \nnorm{\hat{\T}^{\ell+1} \x^\ell - \bar{\x}^{\ell+1}}^2 + \nnorm{\e^\ell}^2 + 2 \diam(\D) \nnorm{\e^\ell}.
\end{align*}
But by assumption, the random coordinate update operator $\hat{\T}^{\ell+1}$ is $\alpha(\ell)$-averaged, hence
\begin{align}
	\nnorm{\x^{\ell+1} - \bar{\x}^{\ell+1}}^2 &\leq \nnorm{\x^\ell - \bar{\x}^{\ell+1}}^2 + \label{eq:intermediate-bound-5} \\ &- \frac{1 - \alpha(\ell)}{\alpha(\ell)} \nnorm{(\I - \T^{\ell+1}) \x^\ell}^2 + \nonumber \\ &+ \nnorm{\e^\ell}^2 + 2 \diam(\D) \nnorm{\e^\ell}. \nonumber
\end{align}
Now, by \cref{as:averaged}, Cauchy-Schwarz inequality, and \cref{eq:norm-equivalence} we have that
$$
    \nnorm{\x^\ell - \bar{\x}^{\ell+1}}^2 \leq \nnorm{\x^\ell - \bar{\x}^\ell}^2 + \frac{\sigma^2}{\pmin} + \frac{2 \diam(\D) \sigma}{\sqrt{\pmin}}
$$
and using this fact into~\cref{eq:intermediate-bound-5} yields
\begin{align}
	\nnorm{\x^{\ell+1} - \bar{\x}^{\ell+1}}^2 &\leq \nnorm{\x^\ell - \bar{\x}^{\ell}}^2 + \label{eq:intermediate-bound-6} \\ &- \frac{1 - \alpha(\ell)}{\alpha(\ell)} \nnorm{(\I - \T^{\ell+1}) \x^\ell}^2 + \nonumber \\ &+ \frac{\norm{\e^\ell}^2 + \sigma^2}{\pmin} + 2 \diam(\D) \frac{\norm{\e^\ell} + \sigma}{\sqrt{\pmin}}. \nonumber
\end{align}

Rearranging~\cref{eq:intermediate-bound-6}, averaging over time, and using the telescopic sum we have
\begin{align*}
    &\frac{1}{\ell+1} \sum_{h = 0}^\ell \frac{1 - \alpha(h)}{\alpha(h)} \nnorm{(\I - \T^{h+1}) \x^h}^2 \leq \\
    &\leq \frac{1}{\ell+1} \Bigg( \frac{\norm{\x^0 - \bar{\x}^0}^2}{\pmin} + \\
    &+ \sum_{h = 0}^\ell \frac{\norm{\e^h}^2 + \sigma^2}{\pmin} + 2 \diam(\D) \frac{\norm{\e^h} + \sigma}{\sqrt{\pmin}} \Bigg).
\end{align*}
Now defining $\bar{\alpha} = \max_{h \in \{ 0, \ldots, \ell\}} \alpha(h)$, and using \cref{eq:norm-equivalence} we have that
\begin{align}
    &\frac{1}{\ell+1} \sum_{h = 0}^\ell \norm{(\I - \T^{h+1}) \x^h}^2 \leq \label{eq:intermediate-bound-8} \\
    &\leq \frac{\pmax}{\pmin} \frac{1}{\ell+1} \frac{\bar{\alpha}}{1 - \bar{\alpha}} \Bigg( \norm{\x^0 - \bar{\x}^0}^2 + \nonumber \\
    &+ \sum_{h = 0}^\ell \norm{\e^h}^2 + \sigma^2 + 2 \diam(\D) \sqrt{\pmin} \left( \norm{\e^h} + \sigma \right) \Bigg). \nonumber
\end{align}

Finally, by \cref{as:stochastic-framework-2} and the properties of sub-Weibull r.v.s (cf. \cref{lem:closure-sub-weibull-class}) we know that the right hand side of \cref{eq:intermediate-bound-8} is a sub-Weibull with parameters $\theta$ and
$$
    \sigma^2 + \nu^2 + 2 \diam(\D) \sqrt{\pmin} (\sigma + \nu)
$$
and by \cref{lem:high-probability-bound} this yields the thesis. \qed

%--------------------------------------------------------------
\subsection{Proof of \texorpdfstring{\cref{pr:no-additive-noise-averaged}}{Proposition~6}}\label{proof:pr:no-additive-noise-averaged}
Setting $\sigma = 0$ and $\norm{\e^\ell} = \0$ in~\cref{eq:intermediate-bound-5} yields
\begin{equation}\label{eq:intermediate-bound-7}
	\norm{\x^{\ell+1} - \bar{\x}}^2 \leq \norm{\x^\ell - \bar{\x}}^2 - u^\ell \frac{1-\alpha}{\alpha} \norm{(\I - \T) \x^\ell}^2,
\end{equation}
which implies $\norm{\x^{\ell+1} - \bar{\x}}^2 \leq \norm{\x^\ell - \bar{\x}}^2$, that is, the operator is \emph{stochastic Fej\'er monotone} \cite{combettes_stochastic_2015,combettes_stochastic_2019}. This means that the fixed point residual $\{ \norm{(\I - \T) \x^\ell} \}_{\ell \in \N | u^\ell = 1}$ is a.s. monotonically decreasing \cite[Theorem~3.2]{combettes_stochastic_2015} (cf. in particular \cite[(3.3)]{combettes_stochastic_2015}).

Therefore, summing~\cref{eq:intermediate-bound-7} over time and using Fej\'er monotonicity we can write
\begin{align*}
	\beta(\ell) \norm{(\I - \T) \x^\ell}^2 &\leq \sum_{h = 0}^\ell u^h \norm{(\I - \T) \x^h}^2 \\ &\leq \frac{\alpha}{1 - \alpha} \norm{\x^0 - \bar{\x}}^2
\end{align*}
where we used the fact that the sequence $\{ \norm{\x^{\ell+1} - \x^\ell} \}_{\ell \in \N | u^\ell = 1}$ has $\beta(\ell)$ non-zero terms. Dividing by $\beta(\ell)$ on both sides we get
$$
	\norm{\x^{\ell+1} - \x^\ell}^2 \leq \frac{1}{\beta(\ell)} \frac{\alpha}{1 - \alpha} \norm{\x^0 - \bar{\x}}^2.
$$
We have now the following fact
\begin{align*}
	&\pr{\frac{1}{\beta(\ell)} \norm{\x^0 - \bar{\x}}^2 \geq \frac{1}{(p - \epsilon) (\ell + 1)} \norm{\x^0 - \bar{\x}}^2} = \\ &= \pr{\beta(\ell) \leq (\ell + 1) (p - \epsilon)} \overset{\text{(i)}}{\leq} \exp\left( - (\ell + 1) D(p - \epsilon || p) \right)
\end{align*}
where (i) holds by Sanov's theorem \cite[Theorem~D.3]{mohri_foundations_2018} (cf. \cite[eq.~(D.7)]{mohri_foundations_2018}), and the thesis follows. \qed

%---------------------------------------------------------------------------------------------------
%---------------------------------------------------------------------------------------------------

\bibliographystyle{IEEEtran} % use IEEEtran.bst style
\bibliography{references}

\vfill

\end{document}